\documentclass[11pt]{amsart}

\usepackage[colorlinks,linkcolor={blue}]{hyperref}

\usepackage{amssymb}
 \usepackage{graphicx} 
 \usepackage{amscd}
 \usepackage{enumerate}
 \usepackage{cancel}

\usepackage[none]{hyphenat}

\usepackage{silence}
\usepackage[colorinlistoftodos, textsize=footnotesize]{todonotes}
\makeatletter
\providecommand\@dotsep{5}
\makeatother
\newcommand\blfootnote[1]{%
  \begingroup
  \renewcommand\thefootnote{}\footnote{#1}%
  \addtocounter{footnote}{-1}%
  \endgroup
}

\newcommand{\e}{^}

\makeatletter
\@namedef{subjclassname@2020}{\textup{2020} Mathematics Subject Classification}
\makeatother

 \theoremstyle{plain}

 \newtheorem{Thm}{Theorem}[section]
 
 \newtheorem{Lemma}[Thm]{\bf Lemma}
 
 \newtheorem{Corollary}[Thm]{\bf Corollary}
 \newtheorem{Theorem}[Thm]{\bf Theorem}
 \newtheorem{Proposition}[Thm]{\bf Proposition}

 \theoremstyle{definition}
 \newtheorem{Definition}[Thm]{\bf Definition}

 \theoremstyle{remark}
 
 \newtheorem{Example}[Thm]{\bf Example}

 \newtheoremstyle{Cl}
  {5pt}
  {3pt}
  {\sl}
  {}
  {\it}
  {:}
  {.5em}
  {}

 \theoremstyle{Cl}

 \def\begincproof{
                  \renewcommand{\proofname}{\it Proof:}
                  \begin{proof}
                 }

 \def\endcproof{
                \renewcommand{\qedsymbol}{$\diamondsuit$}
                \end{proof} 
                \renewcommand{\qedsymbol}{\openbox}
                \renewcommand{\proofname}{\bf Proof:}
               }

\title[Positive topological entropy for the Standard Map]{Positive topological entropy for the Standard Map}

\author{Fernando Oliveira}


\begin{document}

\begin{abstract}
We show that for the standard map family, for all parameter values, except one, the mapping has positive topological entropy.

The proof of this result depends on results about the existence of transverse  homoclinic orbits, which in turn depend on the following result:

Let $S$ be a compact connected orientable surface and $f:S \rightarrow S$ an orientation preserving area preserving $C \e 1$ diffeomorphism of $S$.
Suppose that $U$ is an invariant domain of $S$ such that $fr_S{U}$ has a finite number of connected components.

Let $b$ be a regular ideal boundary point of $U$ which is fixed under the action induced by $f$ on the ideal boundary of $U$, and let $\hat{f}:C(b) \rightarrow C(b)$ be the homeomorphism induced on the corresponding circle of prime ends.
Let $Z(b)$ be the impression of $b$ in $S$ and assume that all fixed points of $f$ in $Z(b)$ are non degenerate. 

If there exists a fixed prime end $e \in C(b)$ then $C(b)$ has a finite number of fixed prime ends and there exists a finite singular covering $ \phi:C(b)\rightarrow Z(b)$, which is a semiconjugacy between the mapping of prime ends on $C(b)$ and the restriction of $f$ to $Z(b)$.
Furthermore, if $p$ is the principal point of a fixed prime end then $p$ is a fixed point of saddle type, and $Z(b)$ is the connected union of finitely many saddle connections and the corresponding saddles.

This result can be seen as a two dimensional analogue of the dynamics of orientation preserving homeomorphisms of the circle with fixed points.
\end{abstract}

\maketitle

\parskip +5pt

\hspace{20mm}  Universidade Federal de Minas Gerais,  Belo Horizonte, Brasil. 

\hspace{40mm} fernando.oliveira.f@gmail.com




\tableofcontents







 
\section{Introduction}

\subsection{Statement of results}
The standard map is a one parameter family of area preserving diffeomorphisms of the two dimensional torus $T \e 2 =\mathbb{R} \e 2/\mathbb{Z}\e 2$ given by 

\begin{center}
$f_{\mu}(x,y)=(x+y+\frac{\mu}{2\pi} \sin (2\pi x),y+\frac{\mu}{2\pi}\sin (2\pi x))$, $\mu \in \mathbb{R}$.
\end{center}

This map is a model for many physical phenomena. It describes Poincaré's surface of section of the kicked rotator, which consists of a stick that is free from gravitational forces, that rotates without friction around an axis located at one of its ends, and is periodically kicked at the other end.
The kicked rotator approximates systems studied in many fields of physics.
For example, circular particle accelerators accelerate particles by applying periodic kicks, as they circulate in the beam tube. Therefore, the structure of the trajectory of the beam can be approximated by the kicked rotator.
For a more detailed account of the development of the problem, see Lazutkin's work on it, and results on the existence of transverse homoclinic orbits for small values of the parameter, see \cite{Ge1999} and \cite{Ge2005}.

Since $f_{\mu}$ and $f_{-\mu}$ are conjugated and $f_0$ is trivial, we only consider parameters $\mu>0$.

For $\mu \neq 0$ there exist two fixed points, $p=(0,0)$ and $q =(\frac{1}{2},0)$.  
$p$ is always a saddle with positive eigenvalues and it is called \textit{the principal fixed point} of $f_\mu$.
$q$ is elliptic if $0 <\mu <4$ and a saddle with negative eigenvalues if $\mu>4$.

The main result of this article is the following.

\textbf{Theorem} 
\textit{If $\mu \neq 4$ then the standard map satisfies the following}:
\begin{enumerate}
    \item \textit{The four branches of $p=(0,0)$ have topologically transverse homoclinic points.
    \item $f_\mu$ has positive topological entropy.
    \item There exist periodic saddles with transverse homoclinic points}.
\end{enumerate}

This is Theorem \ref{entropy for the standard map} and section $5$ is all dedicated to its proof, which is based on the ideas developed in sections $2$, $3$ and $4$, and the use of the Inverse Function Theorem in a neighborhood of $p$.

In section $2$ we develop the basic ideas of the theories of the ideal completion of a surface $S$ and the prime ends compactification of a domain $U \subset S$.
The main idea is to add to $U$ a circle of prime ends corresponding to each "nontrivial" ideal boundary point of $U$, as a surface contained in $S$.

Ideal boundary points of a surface $S$ are also known as the topological ends of $S$, and we denote them by $b(S)$.
The structure of $b(S)$ and the ideal completion of $S$, $B(S)=S\sqcup b(S)$, are nicely described in Proposition \ref{model}. 
$B(S)$ is almost a surface, and except for some points of $b(U)$ that are accumulated by handles or Mobius bands, $B(S)$ is locally a surface.

Let $S$ be a surface and $U$ a domain of $S$. Every ideal boundary point $b \in b(U)$ leaves an impression in $S$, like a mark.
The set \[ Z(b):= \{ x \in S \mid \exists (x_n)  \text{ in } U \text{ with } {x_n} \rightarrow b \text{ in } B(U) \text{ and } x_n\rightarrow x \text{ in } S \} \] is called the \textit{impression} of $b$ in $S$. 
Roughly, it tells the accumulation in $S$ of sequences $(x_n)$ in $U$ that converge to $b$ in $B(U)$. 

Obviously $Z(b)$ is a subset of the frontier of $U$ in $S$, $fr_S U$.

We will be interested in replacing impressions by circles of prime ends, in the case where the impression $Z(b)$ is a compact subset of $S$ with more than one point. 
Is this case $b$ is called a regular ideal boundary point, and we denote the set of such points by $b_{reg} (U)$.

If $b \in b(U)$ and $Z(b)$ is a one point set contained in $S$, then $b$ corresponds to a puncture of $U$ in $S$, and we replace $Z(b)$ by one point.

We will work with domains $U$ such that $fr_S U$ is compact and has a finite number of connected components.
In this case there exist only a finite number of ideal boundary points $b \in b(U)$ such that $Z(b)$ is compact.
This follows from the following result (Proposition \ref{g+1 ideal boundary points with boundary}).

\textbf{Proposition} 
\textit{Let $K$ be a compact subset of a compact connected surface $S$ and $U$ a residual domain of $K$. If $K$ has $m$ connected components then $U$ has at most $m(g+1)$ ideal boundary points, where $g$ is the genus of $S$}.

As far as we know, this upper bound is new. With this result applied to $K=fr_S U$, it is possible to give an elementary proof of Lemma $2.3$ of \cite{Mather1981}, providing an upper bound for the number of regular ideal boundary points of $U$ in arbitrary surfaces.

If $S$ is a connected surface, $U$ is a domain of $S$ such that $fr_S U$ is compact and has a finite number of connected components, then $b_{reg} (U)$ is a finite set. 
$U$ could have infinite genus and any orientability type. 
In this context, it is possible to defined the prime ends compactification $E(U)$ of $U$.
For each $b \in b_{reg} (U)$ we replace $Z(b)$ by a circle of prime ends and obtain a surface with compact boundary $\hat{U}$ and $E(U)$ is homeomorphic to the ideal completion of $\hat{U}$, $B(\hat{U})$. 
Our definition of $E(U)$ is the same as the one given by Mather in \cite{Mather1982}, except that we add the condition that the end cuts be pairwise disjoint \textit{compact} subsets of $S$.
All of Mather's results extend to this more general context.
These results have already been written, but this general approach made the article extremely long.

Therefore we chose to simplify and assume compactness of the initial surface $S$ and smoothness of maps in order to focus on other important ideas.

Let $S$ be a compact surface and $U$ a domain of $S$ such that $fr_S U$ has a finite number of connected components. It follows that $b(U)$ is a finite set (see Proposition \ref{finitely many components}).
In this situation it is possible to define the prime ends compactification of $U$ following Mather \cite{Mather1982}.

We replace the impression $Z(b)$ of each $b \in b(U)$ with a circle of prime ends $C(b)$ or a point, depending on whether $b$ is regular or not.
In this way, the prime ends compactification of $U$, $E(U)$, is obtained as a compact surface with boundary, whose boundary components are the circles of prime ends $C(b)$, where $b$ ranges over all points of the set of regular ideal boundary points of $U$.
Each prime end represents a way of approaching points of $Z(b) \subset fr_S U$, as we move out of $U$ in the direction of $b$.

We would like to start a description of the other relevant result of this paper, which is in section 3.
We study area preserving diffeomorphisms of compact surfaces $f:S\rightarrow S$ and the structure of the frontier $fr_S U$ of invariant domains $U$, such that $fr_S U$ has a finite number of components. The main problem is to describe the structure of $fr_S U$ and the dynamics of $f$ restricted to $fr_S U$.

Consider the simplest case where $U$ is homeomorphic to an open disk. 
In this case $b(U)$ consists of one ideal boundary point, $Z(b)=fr_S U$, $B(U)$ is homeomorphic to a sphere and $E(U)$ is homeomorphic to a closed disk.
We have the following result (Corollary \ref{Corollary 2 the frontier of domains with fixed prime ends}).

\textbf{Proposition}
\textit{Let $S$ be a compact connected orientable surface and $f$ an area preserving orientation preserving $C \e 1$ diffeomorphism of $S$. Let $U$ be an invariant set homeomorphic to an open disk and suppose that all fixed points of $f$ in $fr_S U$ are non degenerate}. 

\textit{If there exists a fixed prime end $e \in C(b)$ then $fr_S U$ is a connected finite union of connections and the corresponding saddle fixed points}.

Each invariant manifold of a saddle has two branches. A \textit{connection} is a branch contained in two invariant manifolds. It connects the saddles (that may coincide).

This result can be seen as a two dimensional analogue of the dynamics of orientation preserving homeomorphisms of the circle with fixed points.
In this case all periodic points are fixed and the arcs between fixed points are connections in the sense that if $x$ belongs to one of these arcs, then $lim_{n \rightarrow \pm \infty}{f \e n (x)}$ are the end points points of the arc.

Of course we have the same results if the homeomorphism on the circle of prime ends has rational rotation number.

Before we proceed, note that a homeomorphism $f:S \rightarrow S$ induces homeomorphisms $f_* : B(U) \rightarrow B(U)$ and $\hat{f}:E(U) \rightarrow E(U)$. If $f$ is orientation preserving, then so is $\hat{f}$.

The second relevant result of this paper is the following (Theorem \ref{The frontier of domains with fixed prime ends}).

\textbf{Theorem}
\textit{Let $S$ be a compact connected orientable surface and $f:S \rightarrow S$ an area preserving orientation preserving $C \e 1$ diffeomorphism of $S$.
Assume that $U$ is an invariant domain of $S$ such that $fr_S{U}$ has a finite number of connected components}. 

\textit{Let $b$ be a regular ideal boundary point of $U$ such that $f_*(b)=b$ and $\hat{f}:C(b) \rightarrow C(b)$ be the homeomorphism on the corresponding circle of prime ends.
Suppose that all fixed points of $f$ in $Z(b)$ are non degenerate}. 

\textit{Assume that there exists a fixed prime end $e \in C(b)$}.

\textit{Then $C(b)$ has a finite number of fixed prime ends and there exists a finite singular covering $ \phi :C(b) \rightarrow Z(b)$, which is a semiconjugacy between the mapping of prime ends on $C(b)$ and the restriction of $f$ to $Z(b)$.}

\textit{Furthermore, if $p$ is the principal point of a fixed prime end then $p$ is a fixed point of saddle type, and $Z(b)$ is the connected union of finitely many saddle connections and the corresponding saddles}.

This is a result about only one circle of prime ends. 
If all ideal boundary points of $U$ are regular and all circles of prime ends of $U$ have a fixed point (or more generally, rational rotation number), then the whole frontier of $U$ is a finite union of connections.

Many articles, which deal with related results, use the hypothesis that elliptic periodic points that belong to $fr_S U$ and arise as principal points of periodic prime ends be Moser stable. 
This was the way to get rid of them, which is no longer necessary: the principal point of a fixed prime end can not be elliptic. 

In section $4$ we present some results about recurrence and the existence of homoclinic points for area preserving diffeomorphisms of compact surfaces. 

If $L$ is an invariant branch of a saddle $p$ and all fixed points of $f$ contained in $cl_{S}L$ are non degenerate, then either $L$ is a connection or $L$ accumulates on both sectors adjacent to itself.
From this we prove the following.

\textbf{Proposition}
\textit{Let $S$ be the sphere $S \e 2$ or the torus $T \e 2$ and $f$ an area preserving orientation preserving $C \e 1$ diffeomorphism of $S$. Let $p$ be a saddle fixed point of $f$ and suppose that the branches of $p$ are invariant sets that are not connections. Assume that all fixed points of $f$ contained in $cl_{S}(W_p \e u \cup W_p \e s)$ are non degenerate}. 

\textit{Then all pairs of adjacent branches of $p$ intersect}.

What we do in Section $5$ is to prove that some of these intersections are topologically transverse.

Note that this proposition and the entire article only deal with fixed points. 
The hypothesis are made only about fixed points and we are not allowed to freely take powers of diffeomorphisms.
We try to make all hypotheses as weak as possible.
This makes results more easily applicable to specific examples.

This work is a continuation of \cite{Ol2}.

In \cite{Ol2} we proved that for all parameter values, except one, the principal fixed point of the Standard Map has homoclinic points. The proof of this fact is simple and takes only a few lines.
The main result of this article is to show that for all parameter values, except one, the Standard Map has \textbf{transverse} homoclinic points.

Theorem \ref{The frontier of domains with fixed prime ends} is the other relevant result of this article. Item $(1)$ of Theorem \ref{The frontier of domains with fixed prime ends} appears in \cite{Ol2}, but it gives information only about fixed points in the frontier of invariant domains. Here we generalize this by giving information on the structure of the whole frontier of invariant domains.

Every time we state a result contained in \cite{Ol2}, we make this very clear in the text.

We also avoid including proofs of new results that appeared in \cite{Ol2}. 
The only place we do this is at the beginning of the proof of Theorem \ref{The frontier of domains with fixed prime ends}.
Instead of including the original proof of item $(1)$ that appeared in \cite{Ol2}, we removed all calculations, long sequences of inequalities and reorganized and summarized the argument, in order to give only a very clear exposition of the structure of the proof. 
The main reason for doing this is that \cite{Ol2} is just a preprint published in the arXiv, where papers do not undergo rigorous peer review.

We avoid text overlap as much as possible.
If this happens, it is only in cases of definitions or summaries of the work of other mathematicians that we need to use.
For example, the first page of subsection $2.6$ is a summary of Mather's theory of prime ends, and a very similar text appears in \cite{Ol2}. 
I was responsible for writing the summary that appeared in \cite{Ol2}.  
If I were to write this summary again, as if for the first time, I would probably write more or less the same, simply because, in my opinion, it is the best way to present Mather's theory.

\subsection{Notation}
In this article, $S$ will denote a smooth connected boundaryless surface, provided with a measure $\mu$, which is finite on compact sets and positive on open sets.
We say that a homeomorphism $f:S \rightarrow S$ is measure preserving if $\mu \big(f \e {-1} (E) \big) = \mu (E)$ for every Borel subset $E$ of $S$.

From a certain point onwards we will assume hypotheses such as the compactness of $S$, its orientability, that $f$ is a $C \e 1$ diffeomorphism and that $f \e * \eta =\eta$, for some non degenerate $2$-form on $S$. In this case, we will still denote by $\mu$ the invariant measure induced by $\eta$.

By a domain we mean an open connected subset of $S$. 
If $A \subset B$, we use the notation $int_{B} (A)$, $cl_{B} (A)$ and $fr_{B} (A)$ for the interior, closure and the frontier of $A$ in $B$, respectively. 
By a disk (open, closed) we mean a set homeomorphic to the disk (open, closed) $D$ in the Euclidean norm of $\mathbb R \e 2$, and by a circle the analogue.
If $K$ is a compact subset of $S$, a connected component of $S \setminus K$ is called a \textit{residual domain} of $K$. 
The boundary of a surface will be denoted by $\partial S$ and we will use $S \e \circ = S \setminus \partial S$  for its interior as a manifold.

Most of the time, we refer to the \textit{connected components} of a topological space simply as \textit{components}.  When referring to an \textit{ideal boundary component}, the three words will always appear together.

\subsection{A note on earlier announcements of Theorem \ref{The frontier of domains with fixed prime ends}}

Theorem \ref{The frontier of domains with fixed prime ends} was announced by this author in 2009/2010 at several conferences and seminars in different countries, including the \textit{International Conference on Dynamical Systems – 2010, Celebrating the 70th birthday of Jacob Palis}, Búzios, Rio de Janeiro, from 02/25 to 03/05, 2010. See \cite{Impa} for information about speakers and titles of their talks at the Conference.
Unfortunately, for family reasons, it is only now that we have been able to write the results.

In 2018, this result was announced again in \cite{KoNa}. 
In their announcement they mention my previous announcement as "a similar result". See Theorem 1.4 and the paragraph following it in \cite{KoNa}.

They state Theorem \ref{The frontier of domains with fixed prime ends} with slightly different assumptions but with exactly the same thesis. 

Their result is stated only for domains of $\mathbb{R} \e 2$ homeomorphic to an open disk, and the generalizations they claim to have are only for maps homotopic to the identity (see the last paragraph of Section 2.4 of  \cite{KoNa}). 

On the other hand, we have no restrictions on the action on homology.

They are never clear about what kind of surfaces and periodic points they can handle.

There is one important point: \textit{as far as we know, so far they have not published the result and have never shown a complete argument or even a preprint}.

We feel like there is no need to dispute authorship. We will be very happy to see \cite{KoNa} published with the result proven with different hypotheses, especially those related to the degeneration of fixed points.





\section{Topological preliminaries}

\subsection{The ideal completion of a surface}

We are going to describe a compactification of $S$ by the addition of its ends or ideal boundary points. See \cite{AhSa}, \cite{Ma}, \cite{Mather1981},  and \cite{Ri}.

We start by considering  decreasing sequences $(P_{n})_{n\geq k}$
of non-empty open connected subsets of $S$ ($P_{n}\supset P_{n+1}\:\forall\:n\geq k)$ such that $fr_{S}P_{n}$ is compact for every $n\geq k$.

Let $(C_{n})$ be a decreasing sequence of subsets of $S$. 
We say
that $(C_{n})$ \emph{leaves compact subsets} of $S$ if for every
compact subset $K$ of $S$ there exists $n_{0}$ such that $K\cap C_{n}=\emptyset$ for $n\geq n_{0}$.

Let $(P_{n})$ be a decreasing sequence of non-empty open connected subsets of a surface $S$ such that $fr_{S}P_{n}$ is compact for
every $n$. 
It is easy to see that the following conditions are equivalent:
\begin{itemize}
\item $(P_{n})$ leaves compact subsets of $S$ .
\item $\cap_{n}cl_{S}P_{n}=\emptyset$.
\end{itemize}

\begin{Definition}
An \emph{ideal boundary component} of $S$ (IBC) is a decreasing sequence
$(P_{n})$ of non-empty open connected subsets of $S$ such that $fr_{S}P_{n}$ is compact for every $n$ and $(P_{n})$ satisfies one of the two equivalent conditions above.
\end{Definition}

The following example shows that there exist decreasing sequences
of open connected subsets $(P_{n})$ such that $fr_{S}P_{n}$ is compact
for every $n$ and $\cap_{n}P_{n}=\emptyset$, but $\cap_{n}cl_{S}P_{n}\neq\emptyset$
and $(P_{n})$ does not leave compact subsets of $S$. This shows
that we can not weaken the condition $\cap_{n}cl_{S}P_{n}=\emptyset$
in the definition of IBCs to $\cap_{n}P_{n}=\emptyset$.

\begin{Example}
Let us consider polar coordinates $(\theta,r)$ on $S=\mathbb{R}^{2}$
and define $P_{n}=\{(\theta,r)\in S\thinspace|\thinspace r>n\}\cup\{(\theta,r)\in S\thinspace|\,0<\theta<\frac{1}{n}\}$.
The sequence $(P_{n})$ is decreasing, $fr_{S}P_{n}$ is compact for
every $n$ , and $\cap_{n}P_{n}=\emptyset$. On the other hand, $\cap_{n}cl_{S}P_{n}=\{(\theta,r)\in S\thinspace|\thinspace r\geq0\;\mathrm{and}\;\theta=0\}$
and $(P_{n})$ does not leave circles centered at $(0,0)$.
\end{Example}

Let $(P_{n})$ and $(Q_{m})$ be ideal boundary components of $S$.
We say that $(P_{n})$ and $(Q_{m})$ are \emph{equivalent} if for
every $n$ there is $m$ such that $P_{n}\subset Q_{m}$ and vice
versa. 
An \emph{end} or \emph{ideal boundary point} of $S$ is an
equivalence class of ideal boundary components. 
The set of ideal boundary points, denoted by $b(S)$, is called the \emph{ideal boundary} or \textit{the set of ends} of $S$, and the disjoint union $B(S):=S\sqcup b(S)$ is called the \emph{ideal completion} or \textit{the ends compactification} of $S$. 

Let $V$ be an open subset of $S$. We define $V'$ as the set of
ideal boundary points of $S$ whose representing IBCs $(P_{n})$ satisfy
$P_{n}\subset V$ for some $n$. Let $V^{*}:=V\sqcup V'$. Obviously $V'\cap W'=(V\cap W)'$ and $V^{*}\cap W^{*}=(V\cap W)^{*}$.
On $b(S)$ and $B(S)$ we are going to consider the topologies generated by the bases of sets formed by $V'$ and $V^{*}$, respectively, where $V$ ranges over the collection of all connected open subsets of $S$ with $fr_{S}V$ compact.

We are going to write $b=[(P_{n})]$ or just $b=(P_{n})$ to say that
$(P_{n})$ is an IBC representing $b\in b(S)$.

A characterization of $B(S)$ as a compactification of $S$ is given by the following result:

\begin{Proposition}

$B(S)$ is a compactification of $S$ that satisfies the following properties:
\begin{enumerate}
        \item $B(S)$ is a locally connected Hausdorff space.
        \item $b(S)$ is totally disconnected.
        \item $b(S)$ is non separating on $B(S)$ (meaning that for any open connected subset $V$ of $B(S)$, $V \setminus b(S)$ is connected). 
\end{enumerate}

If $M$ is another compactification of $S$ that satisfies these properties, then there exists a homeomorphism from $M$ onto $B(S)$ which is the identity on $S$.
\end{Proposition}

For a proof see sections 36 and 37 of chapter 1 of \cite{AhSa}.



\subsubsection{Exhaustions}

Now we would like to describe a practical method for calculating ideal boundary points.

An \textit{exhaustion} of $S$ is an increasing sequence $(F_n)$ of compact connected bordered surfaces contained in $S$, such that for every $n$ we have the following:

\begin{itemize}
        \item $F_n \subset F_{n+1} \e \circ = int_S{F_{n+1}}$.
        \item $S = \cup _{n}F_{n}$.
        \item If $W$ is a connected component of $S \setminus F_{n}$, then $cl_{S}W$ is a non compact bordered surface whose boundary consists of exactly one connected component of $\partial F_{n}$ (one circle).
\end{itemize}

We have the following existence result.

\begin{Proposition}
Every non compact connected surface admits an exhaustion.
\end{Proposition}

For a proof see Theorem of section 29A of chapter 1 of \cite{AhSa}.
Now we would like to state another well known result.

\begin{Proposition} \label{special components}
Let $(F_{n})$ be an exhaustion of a non compact connected surface $S$. 
Then every ideal boundary component of $S$ is equivalent to another of the form $(W_{n})$, where $W_{n}$ is a connected component of $S \setminus F_{n}$.
\end{Proposition}

The proof is not difficult. 
It is a matter of showing that any ideal boundary component of $S$ is equivalent to another, where the sets are components of the complement of the surfaces of an exhaustion.

We would like to emphasise the following simple consequences of Proposition \ref{special components}.

\begin{Corollary}
Every ideal boundary point of $b \in b(S)$ can be represented by an IBC $(W_{n})$, where $fr_S{W_n}$ is a circle.

\end{Corollary}

\begin{Corollary}\label{only finitely many ideal boundary points}
If for every $n$ the number of connected components of $S \setminus F_{n}$ is less than or equal $k$ then $S$ has at most $k$ ideal boundary points.
\end{Corollary}



\subsubsection{Kerékjártó's Theorem}

This is a result that gives necessary and sufficient conditions for connected boundaryless surfaces to be homeomorphic.

We say that $S$ has \textit{finite genus} if there exists a compact bordered surface $K$ contained in $S$ such that $S \setminus K$ is homeomorphic to a subset of the plane. 
In this case the genus of $S$ is defined to be the genus of $K$. 
Otherwise we say that $S$ has \textit{infinite genus}.
The genus of a connected surface may also be defined as is the maximum number of disjoint simple closed curves that can be embedded in the surface without disconnecting it.  

Now we are going to define four types of orientability for a non compact surface.
Assume that $S$ is not orientable. 
$S$ is \textit{finitely non orientable} if there exists a compact bordered surface $K$ contained in $S$ such that $S \setminus K$ is orientable. Otherwise we say that $S$ is \textit{infinitely non orientable}. 
Every compact non orientable surface is the connected sum of a compact orientable surface and one or two projective planes. 
If $S$ is finitely non orientable and $S \setminus K$ is orientable, we say that $S$ is of \textit{odd} or \textit{even} non orientability type according to $K$ is the connected sum of a compact orientable surface and one or two projective planes, respectively.

Let $b=(P_n) \in b(S)$. 
We say that $b$ is \textit{planar} if some $P_n$ is homeomorphic to a subset of the plane. 
Otherwise, $b$ is called non planar, and the set of non planar ideal boundary points of $S$ is denoted by $b'(S)$. 
If $b \in b'(S)$ then the genus of $P_n$ is infinite for every $n$.

We say that $b$ is orientable if some $P_n$ is orientable.
Otherwise, $b$ is called non orientable, and the set of non orientable ideal boundary points is denoted by $b''(S)$.
If $b \in b''(S)$ then $P_n$ is infinitely non orientable for every $n$.

Obviously $b(S) \supset b'(S) \supset b''(S)$.

\begin{Proposition} \label{Kerekjarto}
(Kerékjártó) Let $S_{1}$ and $S_{2}$ be two non compact connected surfaces which have the same genus and the same orientability type. Then $S_{1}$ and $S_{2}$ are homeomorphic if and only if there exists a homeomorphism of $b(S_{1})$ onto $b(S_{2})$, such that $b'(S_{1})$ and $b''(S_{1})$ are mapped onto $b'(S_{2})$ and $b''(S_{2})$, respectively.
\end{Proposition}

For a proof see Kerékjártó \cite{Ke} and Theorem 1 of \cite{Ri}.

\subsubsection{A model for non compact connected boundaryless surfaces}

Let $S^{2} =\mathbb{R}^{2} \cup \{\infty\}$ be the one point compactification of the plane. 
We consider the Cantor ternary set $C$ as the subset of $S^{2}$ consisting of all points $(x,0)$ such that $x$ has one ternary expansion which contains no $1$'s. 

Let $X \supset Y \supset Z$ be subsets of $C$ homeomorphic to $b(S)$, $b'(S)$ and $b''(S)$, respectively (it is possible that $S$ have finite genus which is equivalent to $b'(S)=b''(S)= \varnothing$). 

If we look at $C$ as obtained by the process of removing middle thirds, then $C = \cap _{n \geq 1}J_{n}$  where $(J_{n})$ is a nested sequence of sets consisting of the union of pairwise disjoint closed intervals $I_{n}^{k}$ of length $\frac{1}{3^{n}}$, $1 \leq k \leq 2^{n}$. 
For each $n$ choose a collection of $2^{n}$ pairwise disjoint open balls $B_{n}^{k}$ such that $I_{n}^{k} \subset B_{n}^{k}$, the centers of $B_{n}^{k}$ and $I_{n}^{k}$ coincide, and the balls have the same radius. 
We also want that $B_{n +1}^{l} \cap B_{n}^{k} = \varnothing $ or $B_{n +1}^{l} \subset B_{n}^{k}$. 
We still have that $C = \cap _{n \geq 1} \cup _{1 \leq k \leq 2^{n}}B_{n}^{k}$ and for each $x \in C$ there exists a unique sequence $k_{n}$ such that $ \cap _{n \geq 1}B_{n}^{k_{n}} =\{x\}$. 

Every ball $B_{n}^{k}$ contains exactly two balls $B_{n +1}^{l}$ and $B_{n +1}^{l +1}$. 

If $B_{n}^{k}$ contains a point of $Z$ then we choose a closed disk $D$ contained in $B_{n}^{k}$ and disjoint from the closures of  $B_{n+1}^{l}$ and $B_{n+1}^{l+1}$, and make the connected sum of $S^{2}$ and a projective plane along the boundary of $D$.

If $B_{n}^{k}$ contains a point of $Y -Z$ then we choose a closed disk $D$ contained in $B_{n}^{k}$ and disjoint from the closures of  $B_{n +1}^{l}$ and $B_{n +1}^{l +1}$, and make the connected sum of $S^{2}$ and a torus along the boundary of $D$. 

Let $M_0$ be the set obtained after these connected sums and $M_1 = M_0 \setminus X$. 

In $M_0$, we have that points of $Z$ are accumulated by projective planes, points of $Y \setminus Z$ are accumulated by tori and points of $M_0 \setminus Y$ (which includes $X \setminus Y$) have neighborhoods homeomorphic to open balls.
Therefore $M_1$ is a surface.

For each $x \in X$ there exists a unique sequence $k_{n}$ such that $ \cap_{n \geq 1}B_{n}^{k_{n}} =\{x\}$. 
This gives a one to one correspondence between points of $x \in X$ and all ideal boundary components $(B_{n}^{k_{n}})$ of $b(M_1)$ . 
It is not difficult to show that this correspondence gives a homeomorphism from $b(M_1)$ onto $X$ that takes $b'(M_1)$ to $Y$ and $b''(M_1)$ to $Z$ (the final part of the proof of Theorem 2 of \cite{Ri} contains a detailed proof of this fact).

From the beginning we can assume $S$ and $M_1$ have the same genus and orientability type. It follows from Proposition \ref{Kerekjarto} that $S$ and $M_1$ are homeomorphic. To summarize, we have the following:

\begin{Proposition}\label{model}
Let $S$ be a non compact connected surface of infinite genus. 
Then $S$ is homeomorphic to a surface obtained from the sphere $S^{2}$ by removing a totally disconnected compact subset $X$, taking a collection of pairwise disjoint closed disks $(D_{n})$ and making the connected sum of $S^{2}$ and a torus or a projective plane along the boundaries of $D_{n}$. 
Given a neighborhood $W$ of $X$ in $S^{2}$ all but finite many disks $D_{n}$ are contained in $W$.
\end{Proposition}

See Theorem 3 of \cite{Ri} and the discussion preceding it. 
Although $B(S)$ is not always a surface, the above description gives a very good description of what it looks like. 
$B(S)$ is locally homeomorphic to $\mathbb{R}^{2}$ at every point not in $b'(S)$. Points of $b''(S)$ are accumulated by projective planes and points of $b'(S) \setminus b''(S)$ are accumulated by tori.

Many results become very clear from Proposition \ref{model}. For example.

\begin{Proposition} \label{finite genus}
The following conditions about a connected boundaryless surface are equivalent.
\begin{enumerate}
    \item $B(S)$ is a compact surface.
    \item $S$ has finite genus.
    \item Every $b \in b(S)$ is planar.
\end{enumerate}
\end{Proposition}



\subsection{The impression of an ideal boundary point}

In this subsection $S$ will be a compact connected surface.

We will consider domains $U$ of $S$ and start to investigate the relation between $b(U)$ and $fr_S U$.

Firstly, a simple consequence of Proposition \ref{finite genus}.

\begin{Proposition}
    Let $S$ be a compact connected surface and $U$ a domain of $S$. Then every ideal boundary point of $U$ is planar and $B(U)$ is a compact surface.
\end{Proposition}

Let $b=(P_n) \in b(U)$. 
The \textit{impression} of $b$ in $S$ is the set $Z(b):=\cap_n cl_S {P_n}$. 

\begin{Proposition} \label{different definition of Z(b)}
    Let $b=(P_n) \in b(U)$. $Z(b)$ is the set of points $x \in S$  such that there exists a sequence $(x_n)$ in $U$ such that $x_n \rightarrow b \in b(U)$ in $B(U)$ and $x_n \rightarrow x$ in $S$.
\end{Proposition}

\begin{proof}
Let $x\in S\,|\,\exists\,(x_{n})\text{ in }U\text{ such that }x_{n}\rightarrow b \in b(U) \text{ in }B(U)\text{ and }x_{n}\rightarrow x\text{ in }S$.
For every $k$ there exists $n_{k}$ such that $x_{n}\in P_{k}$ for
$n\geq n_{k}$. Since $x_{n}\rightarrow x\text{ in }S$ we have
that $x\in cl_{S}P_{k}$ for every $k.$

Conversely, assume that $x \in cl_{S}P_{n}$ for every $n$.
Let $(W_{k})$ be a fundamental system of neighborhoods of $x$ in
$S$. For each $k$ take a point $x_{k}\in W_{k}\cap P_{k}.$ Obviously
$x_{k}\rightarrow x$ in $S$. Since $x_{k}\in P_{k}$ for all
$k$ we have that $x_{k}\rightarrow b$ in $U$.
\end{proof}

Now we would like to prove the following.

\begin{Proposition} \label{frontier union of impressions}
    Let $S$ be a compact connected surface and $U$ a domain of $S$. Then $fr_S U = \cup_{b \in b(U)} Z(b)$ (remark: $b(U)$ may be infinite).
\end{Proposition}

\begin{proof}
Let $x \in Z(b)$, where $b=(P_n) \in b(U)$.
    
Since $P_n \subset U$ for every $n$, we have that $Z(b) \subset cl_S U$.
Assume by contradiction that $x \in U$. 
Then $x \in U \cap cl_S P_n = cl_U P_n$ for every $n$, contradicting the fact that $(P_n)$ is an IBC of $U$.
This proves that $x \in fr_S U$.

Conversely, suppose that $x \in fr_S U$.

Let $(x_k)$ be a sequence in $U$ such that $x_k \rightarrow x$ in $S$ and consider an exhaustion $(F_n)$ of $U$. 
Recall that $U\setminus F_n$ has a finite number of components for every $n$.
This implies that, for each $n$, at least one of the components of $U \setminus F_n$ contains an infinite number of elements of $(x_k)$.

\textbf{Claim.} \textit{For every $m$, there exists open connected sets $P_1 \supset P_2 \supset \dots \supset P_m$ such that:
    \begin{enumerate}
        \item $P_j$ is a component of $U \setminus F_j$, for $1 \leq j \leq m$.
        \item $\{k \mid x_k \in P_m\}$ is infinite.
    \end{enumerate} }
    
The claim is obvious for $m=1$.
Assume that the claim is true for $1 \leq j \leq m$. 
Then at least one of the components of $U \setminus F_{m+1}$ contained in $P_m$, say $P_{m+1}$, has the property that $\{k \mid x_k \in P_{m+1}\}$ contains an infinite subset of $\{k \mid x_k \in P_m\}$.
This proves the claim.

From the claim we conclude that $(P_n)$ is an IBC of $U$ that defines an ideal boundary point $b \in b(U)$.
We also conclude that there exists an increasing sequence $k_n \rightarrow \infty$ such that $x_{k_n} \in P_n$ for every $n$. 
We have that $x_{k_n} \rightarrow b$, and since $x_{k_n} \rightarrow x$ in $S$ we have that $x \in Z(b)$.  
\end{proof}



\subsection{Domains whose frontier has a finite number of connected components}

In this subsection $S$ will be a compact surface. We would like to prove the following result.

\begin{Proposition} \label{finitely many components}
Let $S$ be a compact connected surface and $U$ be a domain of $S$. 
Then the following conditions on $U$ are equivalent:
\begin{enumerate}
   \item $U$ is a residual domain of a compact set $K$ such that $fr_S K$ has a finite number of connected components.
   \item $U$ has a finite number of ideal boundary points. 
   \item $fr_S{U}$ has a finite number of connected components.
\end{enumerate}
\end{Proposition}

These conditions are not equivalent if $S$ is not a surface and could be false in more general compact connected topological spaces. 

Consider the following compact connected subset of $\mathbb{R} \e 2$: $X=([0,1] \times \{0,2\}) \cup (C \times [0,2])$, where $C$ is an infinite closed subset of $[0,1]$. 
If $K =\{(x,y) \in X \thinspace \vert \thinspace  y \leq 1\}$ then $K$ is compact and connected. 
$K$ has only one residual domain $U$, but $fr_S{U}$ has an infinite number of components and ends.



\subsection{An upper bound on the number of ideal boundary points of a domain}

Now we would like to give an upper bound on the number of ideal boundary points of a
domain $U$ contained in a compact connected surface $S$. 
This result is known among experts in low-dimensional topology.

This Proposition appeared in \cite{Mather1981} in the context of an arbitrary surface $S$ and a domain $U$ such that $fr_U S$ is compact and has a finite number of components. Mather only proves that $U$ has a finite number of relatively compact ideal boundary points, but does not provide an upper bound for them. 
It also appeared in \cite{Ol3}, but with a different proof that had a gap. 
Our lemma \ref{enough to consider} solves this problem. 

Proposition \ref{g+1 ideal boundary points with boundary} can be used to give an elementary proof of Lemma 2.3 of \cite{Mather1981}, with an upper bound  for the number of relatively compact ideal boundary points explicitly given.

\begin{Proposition} \label{g+1 ideal boundary points with boundary}
    Let $K$ be a compact subset of a compact connected surface $S$ and $U$ a residual domain of $K$. If $K$ has $m$ connected components then $U$ has at most $m(g+1)$ ideal boundary points, where $g$ is the genus of $S$.
\end{Proposition}

Before the proof, we would like to note that the upper bound is sharp in the sense that, on any compact connected surface, there exist examples where it is reached.
In a compact connected orientable surface $S$ of genus $g$, take one meridian $\gamma_i$ around each handle of $S$ and connected them with a simple curve $\alpha$ in the simplest way.
Then $K=\alpha \cup \gamma_1 \cup \dots \cup \gamma_g$ is a compact connected subset of $S$ and $U=S \setminus K$ is connected.
It easy to see that $U$ has $g+1$ ideal boundary points.
In the non orientable case we can give a similar example by thinking of $S$ as the sphere with $g$ Mobius strips attached to it, and taking the curves $\gamma_i$ as the "the central circles" of the strips.

We are going to divide the proof into a few lemmas.

\begin{Lemma} \label{enough to consider}
    We may assume that $S=K \sqcup U$, where $U$ is the only residual domain of $K$.
\end{Lemma}

\begin{proof}
Let $W= \sqcup_{\lambda \in L} V_\lambda$ be the union of the components of $S \setminus K$ different from $U$ and $K' = S \setminus U = K \sqcup W$.
Obviously, $S = K' \sqcup U$. 

Since $K \subset K'$ every component of $K$ is contained in some component of $K'$.

\textbf{Claim.} \textit{Every component of $K'$ contains at least one component of $K$}.

To prove the claim it is enough to show that every component $C$ of $K'$ intersects $K$.
Assume by contradiction that $C \cap K = \varnothing$.
Then $C \subset W$. 
The components of $W$ are the sets $(V_{\lambda})_{\lambda \in L}$.
Therefore $C \subset V_{\lambda_0}$ for some $\lambda_0 \in L$.
But $C$ is a maximal connected subset of $K'$, and therefore $C = V_{\lambda_0}$.
This implies that $C$ is open in $S$.

On the other hand, being a component, $C$ is closed in $K'$ and $C$ is closed in $S$.
This contradicts that $S$ is connected, which proves the claim.

From the claim we conclude that the number of components of $K'$ is less than or equal to the number of components of $K$.
Therefore we could use the decomposition $S = K' \sqcup U$ to prove Proposition \ref{g+1 ideal boundary points with boundary}.
\end{proof}

Now we would like to calculate an upper bound for the number of ideal boundary points of $U$.
Let $(F_{n})_{n \geq 1}$ be an exhaustion of  $U$.

\begin{Lemma} \label{component of S-F_n contains component of K}
    If $W$ is a component of $S \setminus F_{n}$ then $W$ contains at least one component of $K$.
\end{Lemma}

\begin{proof}
Since $K \subset S \setminus F_n$, we have that every component of $K$ is contained in some component of $S \setminus F_n$. 
It is enough to show that $W \cap K \neq  \varnothing $.

Let us assume by contradiction that  $W \cap K = \varnothing $. 

It follows that $W \subset U$ and that $W$ is contained in a connected component $V$ of $U \setminus F_n$. 
But every component of $U \setminus F_n$ is contained in some component of $S \setminus F_n$. From this we conclude that $W=V$. 

We have that $cl_S W = W \cup \sigma_1 \cup  . . . \cup \sigma_l$,  where each $\sigma_i$ is a circle of $\partial{F_n} \subset U$. 
From this we conclude that $cl_S W \subset U$ and that $cl_U V = cl_U W = U \cap cl_S W = cl_S W$ is compact. 
But this contradicts the fact that $(F_n)$ is an exhaustion of $U$.

Therefore $W$ $\cap$ $K$ $\neq$ $\emptyset$ and every connected component of $S \setminus F_n$ contains at least one  component of $K$.
\end{proof}

Let $E_1,...,E_k$ be the components of $S \setminus F_n$. 
From Lemma \ref{component of S-F_n contains component of K} we have that $k \leq m$.

For $1\leq i\leq k$, we have that each $E_i$ is an open set whose frontier in $S$ is a union of circles of $\partial F_n$.
We write $fr_S E_i = C_{i1} \sqcup \dots \sqcup C_{i \nu_{i}}$, where $\nu_i \geq 1$.

Since $\partial F_n = fr_S F_n = fr_S E_1 \sqcup \dots \sqcup fr_S E_k$, we have that the number of components of $\partial F_n$ is $\nu_1 + \dots + \nu_k$.

With this notation we have that
\begin{center}
    $S = int_S F_n \sqcup fr_S F_n \sqcup (E_1 \sqcup \dots \sqcup E_k)$
    
\end{center}
or
\begin{center}
    $S = int_S F_n \sqcup \big( \sqcup_{1 \leq i \leq k}(E_i \sqcup fr_S E_i) \big)$.
\end{center}

\begin{Lemma} \label{g+2}
    For each $i$, $1 \leq i \leq k$, and any subset $\{j_1, \dots, j_s\}$ of $\{1, \ldots ,\nu_i\}$  we have that the set $int_SF_n \sqcup E_i \sqcup \big(C_{i j_1} \sqcup \dots \sqcup C_{i j_s}\big)$
is connected.
\end{Lemma}

\begin{proof}
We have that $E_i$ is connected and $E_i \subset E_i \sqcup \big(C_{i j_1} \sqcup \dots \sqcup C_{i j_s} \big) \subset cl_S E_i$.
Therefore $E_{i}\sqcup\left(C_{ij_{1}}\sqcup\dots\sqcup C_{ij_{s}}\right)$is connected.

Similarly, $int_S F_n$ is connected, and since 
\begin{center}
    $int_S F_n \subset int_S F_n \sqcup (C_{i j_1} \sqcup \dots \sqcup C_{i j_s}) \subset cl_S F_n$
\end{center}
we have that $int_S F_n \sqcup (C_{i j_1} \sqcup \dots \sqcup C_{i j_s})$ is connected.
\end{proof}

We claim that $\nu_i \leq g+1$, for $1 \leq i \leq k$.

If we had $\nu_l \geq g+2$ for some $l \in \{1,\dots,k\}$,
then $S \setminus (C_{l1} \sqcup \dots \sqcup C_{l(g+1)})$ would be disconnected.
But $S \setminus (C_{l1}\sqcup\ldots\sqcup C_{l(g+1)})=$ 
\begin{center}
    $\big[ \bigsqcup_{1 \leq i \leq k,i \neq l} \big (int_S F_n \sqcup (E_{i}\sqcup fr_S E_i) \big) \big] \bigsqcup \big[int_S F_n \sqcup (E_{l}\sqcup C_{l(g+2)} \sqcup \dots \sqcup C_{l\nu_{l}})\big]$
\end{center}
and the second side of the equality is a connected set by Lemma \ref{g+2}, a contradiction.

The number of components of $\partial F_{n}$ is $\eta_{n}:=\nu_{1}+...+\nu_{k}$.
It follows that $\eta_{n}=\nu_{1}+...+\nu_{k}\leq k(g+1)\leq m(g+1)$.

From Corollary $\ref{only finitely many ideal boundary points}$ we have that $U$ has at most $m(g+1)$ ideal boundary points.



\subsection{Proof of Proposition \ref{finitely many components}: equivalent conditions on the frontier of a domain}

Let $S$ be a compact connected surface and $U$ be a domain of $S$. 
We would like to prove that the following conditions on $U$ are equivalent:
\begin{enumerate}
   \item $U$ is a residual domain of a compact set $K$ such that $fr_S K$ has a finite number of connected components.
   \item $U$ has a finite number of ideal boundary points. 
   \item $fr_S{U}$ has a finite number of connected components.
\end{enumerate}

$(1)$ implies $(2)$.

If $U$ is a residual domain of a compact set $K$ such that $fr_S K$ has a finite number of connected components, then by Proposition \ref{g+1 ideal boundary points with boundary} we have that $b(U)$ is finite.

$(2)$ implies $(3)$.

From Proposition \ref{frontier union of impressions} we know that $fr_S U = \cup_{b \in b(U)} Z(b)$.
Since $b(U)$ is finite and each set $Z(b)$ is connected, we have that $fr_S U$ has a finite number of components.

$(3)$ implies $(1)$.
If we take $K$ as $fr_S U$, the result follows from the simple fact that $U$ is always a residual domain of $fr_S U$.



\subsection{The prime ends compactification}

In this subsection we will consider a compact connected surface $S$ and a domain $U \subset S$ such that $fr_S U$ has a finite number of connected components.

We will give a brief description of the theory of prime ends following Mather \cite{Mather1981} and \cite{Mather1982}. See also \cite{Cali}, \cite{FrLe} and \cite{Ol2}.
All results here are well known. They are all easy to prove or a proof can be found in \cite{Mather1982}.

As we saw in Proposition \ref{finitely many components}, $fr_S U$ has finite number of connected components if and only if $b(U)$ is finite.
Observe that $U$ is a surface finite genus and by Proposition \ref{finite genus} $B(U)$ is a compact surface.
Therefore we are in the setting of Mather's work \cite{Mather1982} and we will make a brief presentation of his theory of prime ends.
The basis of Mather's work are \cite{Ca1} and \cite{Ca2}.

The basic idea is the following.
Let $b \in b(U)$.
If $Z(b)$ contains more that one point, then $b$ is called \textit{regular}.
In this case we replace $Z(b)$ with a circle $C(b)$.
If $Z(b)$ contains just one point, then $b$ represents a puncture of $U$ in $S$, and we fill in this puncture with the corresponding point of $S$.
In this way we obtain a compact connected surface with boundary.

A \textit{chain} is a sequence $(V_i)_{i \geq 1}$ of open connected subsets of $U$ such that

\begin{enumerate}
    \item $V_i \supset V_{i+1}$ for every $i$.
    \item $fr_U V_i$ is nonempty and connected for every $i$.  
    \item $cl_S (fr_U V_i) \cap cl_S (fr_U V_j) =\varnothing $ for $i \neq j$. 
\end{enumerate}

A chain $(W_j)$ \textit{divides} $(V_i)$ if for every $i$ there exists $j$ such that $W_j \subset V_i$. 
Two chains are \textit{equivalent} if each divides the other. 
A chain is \textit{prime} if any chain which divides it is equivalent to it.
A sufficient condition for a chain $(V_i)$ to be prime is that there exists $p \in S$ such that $fr_U V_i \rightarrow p$ (every neighborhood of $p$ contains all but finitely many of the sets $fr_U V_i$). 

A \textit{prime point} is an equivalence class of prime chains. 

Let $x \in U$ and consider a family of closed disks $D_1 \supset D_2 \supset \dots $ in $U$ such that $D_{i+1} \subset int_U D_i$ and $ \cap D_i = \{x\}$. 
Then $ (int_U D_i)$ is a prime chain that defines a prime point denoted by $\omega(x)$. 

A \textit{prime end} is a prime point which is not of the form $\omega (x)$ for any $x \in U$.
The set of prime points of $U$ is denoted by $E(U)$. 

We consider a topology on $E(U)$ defined as follows. 
Let $V$ be an open subset of $U$ and denote by $V^{\prime}$ the set of prime points of $U$ whose representing chains are eventually contained in $V$. 
The collection of sets $V^{\prime} \subset E(U)$ such that $V$ is an open subset of $U$ is a basis of a topology on $E(U)$. 

The function $\omega:U \rightarrow E(U)$ is a homeomorphism from $U$ onto an open subset of $E(U)$, and we simply identify $U$ with $\omega(U)$. 

For $e \in E(U)$ let  $\alpha(e) = \cap_i cl_{B(U)}V_i$, where $(V_i)$ is a chain representing $e$. 
Then $\alpha(e)$ consists of one point and $\alpha:E(U) \rightarrow B(U)$ is a continuous function whose restriction to $U$ is the inclusion (or more precisely, $\alpha  \circ \omega$ is the identity on $U$). 

$E(U)$ is a compact connected surface with boundary.  
Let $b_{reg}(U)$ denote the set of regular ideal boundary points of $U$.
We have that $\partial E(U)=\cup_{b \in b_{reg} (U)} \alpha \e {-1} (b)$.
When $b \in b_{reg} (U)$, we have that $\alpha^{-1}(b)$ is homeomorphic to a circle, called the \textit{Caratheodory circle of prime ends associated with} $b$. 
We also denote it by $C(b)$.

We have a decomposition $E(U)=U \sqcup \alpha \e {-1} \big(b(U)\big)$, where $\alpha \e {-1} \big(b(U)\big)$ is the set of all prime ends of $U$.
If $b \in b(U)$ is not regular, then the prime end $\alpha^{-1}(b)$ is just a point of $E(U) \setminus \partial E(U)$.

Let $e \in E(U)$ and $(V_i)$ be a chain representing $e$. 
The set $Y(e) = \cap_i cl_S (V_i)$ is called the \textit{impression of }$e$. 
The definition does not depend on the representing chain, and $Y (e)$ is a compact, connected, non-empty subset of $S$.

We say that $p \in S$ is a \textit{principal point} of $e$ if there is a chain $(V_i)$ which represents $e$ such that $fr_U V_i \rightarrow p$. 
The set of principal points of $e$ is called the \textit{principal set} of $e$ and is denoted by $X(e)$. 
It is a non-empty, compact, connected subset of $S$, and $X(e) \subset Y(e) \subset Z(b)$, where $b = \alpha(e)$.

Let $\beta:(0,1] \rightarrow U$ be a path such that $\underset{t \rightarrow 0}{\lim } \thinspace \beta(t)=p$ in $S$ for some $p \in fr_S U$. 
Then there exist $e\in E(U)$ and $b \in b(U)$ such that $\underset{t \rightarrow 0}{\lim } \thinspace \beta(t)=e$ in $E(U)$ and $\underset{t \rightarrow 0}{\lim } \thinspace \beta(t)=b$ in $B(U)$.

If $X(e) =\{p\}$ for some $p \in fr_S U$, we say that $e$ is an \textit{accessible prime end}. 
We have that $e$ is an accessible prime end if and only if there exists a path $\beta:(0,1] \rightarrow U$ such that $\underset{t \rightarrow 0}{\lim }\beta(t) =p$ in $S$ and $\underset{t \rightarrow 0}{\lim}\beta(t) = e$ in $E(U)$. 

Next we would like to present a result that is used in the proof of Theorem \ref{The frontier of domains with fixed prime ends}. It says that accessible prime ends can be represented by chains whose frontiers are contained in arcs of circles. 

We are going to consider $\mathbb{R}^{2}$ equipped with a norm $\parallel \hspace{1.5mm} \parallel$ and denote by $B_r$, $B_r^{o}$ and $C_r$ the closed ball, open ball and the circle with center at $(0,0)$ and radius $r$, respectively. 

\begin{Lemma} \label{preparation lemma}
Let $S$ be a compact connected surface, $U \subset S$ a domain whose frontier contains a finite number of components. 
Let $e$ be an accessible prime end of $U$ and $p$ its principal point. 
Then there exists $\delta>0$ such that for any decreasing sequence $(r_n)_{n \geq 1}$ contained in $(0,\delta )$ with $\underset{n \rightarrow \infty }{\lim }r_n = 0$, there exists a chain $(V_n)$ representing $e$ such that $fr_U V_n \subset C_{r_n}$.
\end{Lemma}

The proof appears in \cite{Ol2} and is inspired in the proof of Theorem $16$ of \cite{Mather1982}, where Mather consider sequences of piecewise linear disks.

Finally, if $U$ is an invariant domain of $f$, we would like to talk about the actions induced on $B(U)$ and $E(U)$ by the restriction of $f$ to $U$.

We have that $f$ maps an ideal boundary component of $U$ to another ideal boundary component of $U$. 
Therefore the restriction of $f$ to $U$ extends to a homeomorphism $f_*:B(U) \rightarrow B(U)$ and the points of $b(U)$ are periodic under $f_*$.
Since $f$ maps irreducible chains to irreducible chains, the restriction of $f$ to $U$ also extends to a homeomorphism $\hat{f}:E(U) \rightarrow E(U)$ and the circles of prime ends of $U$ are all periodic under $\hat{f}$.

Finally we would like to present a version of Cartwright and Littlewood's fixed point theorem, see \cite{Cali}. 

\begin{Lemma} \label{frontiers of V_i and fV_i intersect for i>=i_0}
Let $b \in b(U)$ be a regular ideal boundary points of $U$ and assume that $\hat{f}\big(C(b)\big)=C(b)$.
Suppose that $e \in C(b)$ is a fixed prime end of $U$, $p$ a principal point of $e$ and $(V_i)$ a chain representing $e$ such that $fr_{U}(V_i) \rightarrow p$. 
There exists $i_0$ such that $fr_{U} (V_i) \cap fr_{U} \big(f(V_i)\big) \neq \varnothing $ for $i \geq i_0$.
\end{Lemma}

The proof is well known and simple. Once the chains $(V_i)$ and $(fV_i)$ are equivalent, we have that the sets $(V_i)$ and $(fV_i)$ intersect for every $i$ sufficiently large and the measure preserving property of $f$ implies that $fr_{U} (V_i) \cap fr_{U} (fV_i) \neq \varnothing $.

\begin{Proposition} \label{Cartwright and Littlewood's fixed point theorem}
Let $e$ be a fixed prime end of $U$ and $p$ a principal point of $e$. Then $f(p)=p$. 
\end{Proposition}

Let $(V_i)$ be a chain defining $e$ such that $fr_{U} (V_i) \rightarrow p$. 
From Lemma \ref{frontiers of V_i and fV_i intersect for i>=i_0} there exists $i_0$ such that for $i \geq i_0$ there exists a point $x_i \in fr_{U} (V_i)$ such that $f(x_i)$ also belongs to $fr_{U} (V_i)$. 
Since $fr_{U} (V_i) \rightarrow p$, we have $x_i \rightarrow p$ and $f(x_i) \rightarrow p$ implying that $f(p)=p$.



\section{The structure and dynamics of the frontier of domains with fixed prime ends}

From now on, $S$ will be a connected compact orientable surface and $f:S \rightarrow S$ an orientation preserving $C\e 1$ diffeomorphism of $S$ such that $f \e * \eta =\eta$, for some non degenerate $2$-form on $S$. 
We will still denote by $\mu$ the invariant measure induced by $\eta$. 
$U$ will be an $f$-invariant domain of $S$ such that $fr_S U$ is compact and has a finite number of connected components.

We would like to carefully define \textit{sectors}, \textit{containment of sectors} and \textit{accumulation on sectors} of saddles. For this we need to establish some notation.

Let $p$ be a fixed point of $f$. 

We say that $p$ is \textit{non degenerate} if $1$ is not an eigenvalue of $df_p$.

A non degenerate fixed point $p$ is \textit{of saddle type} or just a \textit{saddle} if the eigenvalues of $df_p$ are real numbers $\lambda$ and $\lambda \e {-1}$.

The \textit{stable} and \textit{unstable invariant manifolds} of $p$ are defined as:

\begin{center}
$W_{p}^{s} =\{x \in S \thinspace \vert \thinspace  \lim _{n \rightarrow \infty }f^{n}x =p\}$    
and 
\end{center}

\begin{center} $W_{p}^{u} =\{x \in S \thinspace \vert \thinspace  \lim _{n \rightarrow  -\infty }f^{n}x =p\}$, respectively.
\end{center}
$W_{p}^{s}$ and $W_{p}^{u}$ are injectively immersed connected curves.  
The components of $W_{p}^{s} -\{p\}$ and $W_{p}^{u} -\{p\}$, with respect to the topology induced by a parameterization, are called \textit{branches}. 
By a \textit{connection} we mean a branch which is contained in the intersection of two invariant manifolds (possibly of two different fixed points). 

Let $p$ be a saddle of $f$ and $\mathbb{V}$ be the collection of all open neighborhoods $V$ of $p$, where there exist continuous coordinates with $p$ at the origin and in which $f(x,y) =(\lambda x ,\lambda \e {-1} y)$, where $\lambda,\lambda \e {-1} \in \mathbb{R}$.

The components of $W_{p}^{s} \cap V$ and $W_{p}^{u} \cap V$ which contain $p$ are called the \textit{local invariant manifolds of $p$ defined by $V$}. 
The \textit{local branches of $p$ defined by $V$} are the components of the complement of $p$ in the local invariant manifolds defined by $V$. 

The set $\widetilde{V}=\{(x,y) \in V \mid x \neq 0 \text{ and } y \neq 0\}$ and its components are open in $S$.
There are four components of $\widetilde{V}$ that contain $p$ in their closures in $S$.
We will call these sets the four \textit{sectors of p defined by V} and will use $\Sigma_i (V)$, $1 \leq i \leq 4$, to denote them.

$\Sigma_1 (V)$ will be the component where points have coordinates $x>0$ and $y>0$,
$\Sigma_2 (V)$ the component where  $x<0$ and $y>0$, $\Sigma_3 (V)$ the component where  $x<0$ and $y<0$ and $\Sigma_4 (V)$ the component where $x>0$ and $y<0$.

If $V_1,V_2 \in \mathbb{V}$ then $\Sigma_i (V_1) \cap \Sigma_j (V_2) =\varnothing$ if $i \neq j$, or $\Sigma_i (V_1)$ $\sim_p$ $\Sigma_i (V_2)$ define the same germ at $p$.

A \textit{sector of $p$}, denoted by $\Sigma_i(p)$ or just $\Sigma_i$, is the germ at $p$ of the sectors of p defined by V, $\Sigma_i (V)$, where $V \in \mathbb{V}$. $p$ has four sectors, $\Sigma_i(p)$, $1 \leq i \leq 4$.
Sometimes we will simplify writing by referring to the sectors through their representatives.

A set $\mathcal{A}$ \textit{contains a sector} $\Sigma_i(p)$, if $\mathcal{A}$ contains a sector $\Sigma_i(V)$ that represents $\Sigma_i(p)$, for some $V \in \mathbb{V}$.

Now we would like to say what it means for a set $\mathcal{B}$ to accumulate on $p$
with points coming through a sector $\Sigma_i(p)$.
A set $\mathcal{B}$ \textit{accumulates on a sector $\Sigma_i(p)$}, if $\big(cl_S\mathcal{B}\big) \cap \big(\Sigma_i(V)\big) \neq \varnothing$ for every $V \in \mathbb{V}$. 

We say that a \textit{branch $L$ and a sector $\Sigma $} are \textit{adjacent} if 
a local branch of $L$ is contained in the closure of $\Sigma $ in $S$. 
\textit{Two branches are adjacent} if they are adjacent to a single sector. 
\textit{Two sectors are adjacent} if a local branch is contained within the closure of each.

We are going to use the notation ${o}(x)=o(x,f)$ for the orbit of $x$ under $f$ $\{f \e n {(x)} \in S \mid n \in \mathbb{Z}\}$.

For a branch or invariant manifold $L$ and $q_1,q_2 \in L$ we will denote the closed arc of points of $L$ between $q_1$ and $q_2$ by $L[q_1,q_2]$. The open and half-open arcs will be denoted by $L(q_1,q_2)$, $L[q_1,q_2)$ and $L(q_1,q_2]$, respectively.

We say that a fixed point $p$ is \textit{elliptic} if $df_p$ is a rotation by an angle $\theta \neq 0$. 
If $p$ is a point of period $\tau$ we use $f \e \tau$ to define these concepts for $p$.

In this section we will prove the following result.

\begin{Theorem} \label{The frontier of domains with fixed prime ends}

Let $S$ be a compact connected orientable surface and $f:S \rightarrow S$ an area preserving orientation preserving $C \e 1$ diffeomorphism of $S$.
Assume that $U$ is an invariant domain of $S$ such that $fr_S{U}$ has a finite number of connected components. 

Let $b$ be a regular ideal boundary point of $U$ such that $f_*(b)=b$ and $\hat{f}:C(b) \rightarrow C(b)$ the homeomorphism on the corresponding circle of prime ends.
Suppose that all fixed points of $f$ in $Z(b)$ are non degenerate. 

Assume that there exists a fixed prime end $e \in C(b)$. 
Then we know the following.

\begin{enumerate}
\item If $p$ is the principal point of $e$ then $p$ is also a fixed point of $Z(b)$ and $p$ is a saddle.
\item The mapping $\hat{f}:C(b) \rightarrow C(b)$ is orientation preserving and has a finite number of fixed points.
\item There exists a finite singular covering $ \phi :C(b) \rightarrow Z(b)$, which is a semiconjugacy between the mapping of prime ends on $C(b)$ and the restriction of $f$ to $Z(b)$.
\item The impression $Z(b)$ is the connected union of a finite number of saddle connections and the corresponding saddles.
\end{enumerate}

\end{Theorem}

The proof of Theorem \ref{The frontier of domains with fixed prime ends} will be given in the following subsections of Section $3$.

\subsection{The principal point of an accessible fixed prime end is of saddle type}

In this subsection we give a summary of the proof of item $(1)$.

We will only present the important ideas. 
We removed all calculations, long sequences of inequalities and reorganized and summarized the argument, in order to give only a very clear exposition of the structure of the proof.
A complete proof with all calculations is in \cite{Ol2}.
There, interested readers will be able to find all details.

If $p$ is a principal point of a fixed prime end $e \in C(b)$, then by Proposition \ref{Cartwright and Littlewood's fixed point theorem} we know that $f(p)=p$.
Our hypothesis that fixed points of $f$ contained in $Z(b)$ are non degenerate implies that they are isolated fixed points. 
Since the principal set of $e$ is connected and all principal points are fixed, we have that $X(e)$ consists of a one point set $\{p\}$ and $e$ is accessible.

We would like to show that $p$ is of saddle type.
We will assume that $p$ is elliptic and obtain a contradiction.

We will write coordinates in the plane as a complex numbers.
Let $V$ be a neighborhood of $p$ in $S$ where there exist coordinates with $p$ at the origin and in which $f'(0)=e \e {i\alpha}$ with $0<\alpha<2 \pi$.

The idea resembles the role of the minimal invariant circles of a Moser stable fixed point in $fr_S U$ used by Mather.
We consider a chain $(V_i)$ that represents $e$ and such that $\beta_i = fr_U V_i$ are arcs of circles with center at $0$. 
If the arcs are very small, then the dynamics of $f$ and of $f'(0)$ are close enough, so that a finite number of arcs $f \e k (\beta_i)$ rotate around $p$ and satisfy $f \e k (\beta_i) \cap f \e {k+1} (\beta_i) \neq \varnothing$ for $0 \leq k \leq n-1$ and $f \e n (\beta_i) \cap \beta_i \neq \varnothing$.
In this way they close a curve contained in $U$ with $p$ inside it. 

Then it is easy to arrive at a contradiction.


\vspace{2mm}

\textbf{1) Estimating the distance from iterates of $z$ under $f$ and $f'(0)$.}
Let $B_{\delta } \e \circ$ be the open ball of radius $\delta$ centered at that origin with $\delta$ small enough so that $B_{\delta } \e \circ \subset V$ and $f(B_{\delta } \e \circ) \subset V$. 
For points $z$ close to $0$ we will need to estimate the distance between iterates of $z$ under $f$ and $f'(0)$.

Firstly, we would like to make an observation about complex numbers

Let $w$ and $w' \in \mathbb{C}$ and $\epsilon < 1$ be such that $\left \vert w'-w \right \vert  <\epsilon \left \vert w \right \vert $. 
If  $w=re \e {i\eta }$ and $r' =\left \vert w' \right \vert $ then there exists a unique real number $\eta' $ such that 
\begin{enumerate}
    \item $w'=r'e \e {i\eta'}$.
    \item $\left \vert r'-r \right \vert < \epsilon \hspace{0.5mm} r$.
    \item $\left \vert \eta'-\eta \right \vert < \frac{\pi}{2} \hspace{0.5mm} \epsilon $.
\end{enumerate}

For $\epsilon\in(0,1)$, let $\delta>0$ be such that if $\left\vert z \right\vert<\delta$ then $\left\vert f(z)-f'(0)z \right\vert<\epsilon\left\vert z \right\vert$.

From our previous observation about complex numbers, it is easy to see the following.

\begin{Lemma} \label{first estimate between f(z) and f'(0)z} 
Let $z =re \e {i\theta}$ with $r<\delta$ and $r'=\left\vert f(z) \right\vert$. 
Then there exists a unique real number $\theta'$ such that
\begin{enumerate}
    \item $f(z)=r'e \e {i\theta'}$.
    \item $\left\vert r'-r \right\vert<\epsilon r$.
    \item $\left\vert \theta'-(\theta+\alpha) \right\vert<\frac{\pi}{2}\epsilon$.
\end{enumerate}
\end{Lemma}

From this, we conclude that the mapping $F:\mathbb{R} \times (0,\delta ) \rightarrow \mathbb{R} \times (0,+\infty )$ defined by $F(\theta,r)=(\theta',r')$ is a lifting of $f:B_{\delta }-\{0\} \rightarrow \mathbb{C}$ with the following properties: 
\begin{itemize}
    \item $\left \vert r'-r \right \vert < \epsilon r$.
    \item $\left \vert \theta'-(\theta+\alpha) \right \vert <\frac{\pi}{2}\epsilon$
\end{itemize}
for every $(\theta,r) \in \mathbb{R} \times (0,\delta)$.
  
We are going to use $(\theta_j,r_j) =F \e {j} (\theta_0,r_0)$ to denote iterates of a point $(\theta_0,r_0)$.

\vspace{2mm}

\textbf{2) The construction.}
Let $B \subset V$ be a closed ball centered at $0$ such that $Z(b)$ is not contained in $B$.

Let $n$ such that $n \alpha >2 \pi$ and $n (2\pi - \alpha ) > 2 \pi$.

Let $\epsilon \in (0,1)$ such that 
\begin{equation} \label{condition on epsilon 1}
    n \frac{\pi}{2} \epsilon < n \alpha - 2 \pi,
\end{equation}
\begin{equation} \label{condition on epsilon 2}
    n \frac{\pi}{2} \epsilon < n (2 \pi - \alpha ) - 2 \pi, 
\end{equation} 
\begin{equation} \label{condition on epsilon 3}
    \alpha + \frac{\pi}{2} \epsilon <2 \pi \text{  and  } \alpha - \frac{\pi}{2} \epsilon >0.
\end{equation}

Let $\delta>0$ be such that
\begin{enumerate}
    \item $B_{\delta} \subset B \e \circ$.
    \item $f(B_{\delta } \e \circ) \subset V$.
    \item If $\left \vert z \right \vert < \delta$ then $\left \vert f(z)-f'(0)z \right \vert  <\epsilon \left \vert z \right \vert$.
\end{enumerate}

Suppose that the coordinate $r_0$ of the initial point is a positive number that satisfies $r_0 (1+\epsilon ) \e n < \delta$. 

Then from item $(2)$ of Lemma \ref{first estimate between f(z) and f'(0)z} we have $r_j<r_0 (1 +\epsilon) \e j <\delta$ for $1 \leq j \leq n$.
Therefore, if
\begin{equation} \label{r_0}
    r_0 < \frac{\delta}{(1+\epsilon) \e n}
\end{equation}
then the iterates $(\theta_j,r_j) =F \e {j} (\theta_0,r_0)$ are well defined for $1 \leq j \leq n$. 

From item $(3)$ of Lemma \ref{first estimate between f(z) and f'(0)z} we have that 
\begin{equation} \label{estimate fˆj(p) and F'ˆj(0)}
    \left \vert \theta_j -(\theta_0 + j \alpha) \right \vert <j \frac{\pi}{2} \epsilon \hspace{2mm} \text{for} \hspace{2mm} 1 \leq j \leq n.
\end{equation}

From Lemma \ref{preparation lemma}, we know that there exists $\delta_1 >0$ such that for any sequence $(\rho_i)$ contained in $(0,\delta_1)$ such that $lim_{i \rightarrow \infty} \rho_i = 0$, there exists a chain $(V_i)$ representing $e$ such that $fr_{U}V_i \rightarrow p$ and the sets $\beta_i := fr_{U}V_i$ are arcs of circles of radius $\rho_i$ and center at $0$.

There exists $i_1$ such that $\beta_i \subset C_{\rho_i}$ with $\rho_i < \frac{\delta}{(1+\epsilon) \e n}$ for $i \geq i_1$.

If we apply Lemma \ref{frontiers of V_i and fV_i intersect for i>=i_0} to $f$ and $f \e n$ at the same time, we conclude that there exists $i_2$ such that  $f(\beta_i) \cap \beta_i \neq \varnothing$ and $f \e n (\beta_i) \cap \beta_i \neq \varnothing$ for $i \geq i_2$.

Therefore, if we choose $\beta$ to be one of the arcs $\beta_i$ with $i \geq \max\{i_1,i_2\}$, then $\beta$ satisfies the following.

\begin{enumerate}
    \item If the radius of $\beta$ is $r_0$, then $r_0$ that satisfies equation (\ref{r_0}), and the first $n$ iterates of $\beta$ by $F$ are well defined.
    \item $f \e {l}(\beta) \cap f \e {l+1}(\beta ) \neq \varnothing$ for $0 \leq l-1 \leq n$ and $f \e n(\beta) \cap \beta \neq \varnothing$.
\end{enumerate}


\vspace{2mm}

\textbf{3) The iterates of $\beta$ turn around $p$ at least once to close a curve contained in $U$ with $p$ inside.}
Let $\hat{\beta}=(a,b) \times \{r_0\}$ be a lifting of $\beta$ with $0 \leq a < 2 \pi$. 
We have $b-a < 2\pi$. 

Using equations  (\ref{condition on epsilon 3}) and (\ref{estimate fˆj(p) and F'ˆj(0)}) 
it is possible to show the following.

\begin{Lemma} \label{two cases}
$F\hat{\beta} \cap \hat{\beta} \neq \varnothing$ or $F\hat{\beta} \cap (\hat{\beta} + (2\pi,0)) \neq \varnothing $.
\end{Lemma}

To prove this, one uses the fact that that $F \hat{\beta } \cap (\hat{\beta } +(2 k \pi  ,0)) \neq \varnothing $ for some $k \in \mathbb{Z}$ and shows that $k =0$ or $1$. 

The next step is to use equations (\ref{condition on epsilon 1}) and (\ref{estimate fˆj(p) and F'ˆj(0)}) to prove the following.

\begin{Lemma} \label{first case of the lifting}
$F \e n \hat{\beta} \cap (\hat{\beta} +(2k\pi,0)) \neq \varnothing$ for some $k \geq 1$.
\end{Lemma}


\vspace{2mm}

\textbf{4) The proof when $F\hat{\beta} \cap \hat{\beta} \neq \varnothing$.}

As we have seen in Lemma \ref{two cases}, we have two possibilities: $F\hat{\beta} \cap \hat{\beta} \neq \varnothing$ or $F\hat{\beta} \cap (\hat{\beta} + (2\pi,0)) \neq \varnothing$.
We will finish proving item $(1)$ of Theorem \ref{The frontier of domains with fixed prime ends} in the case where $F\hat{\beta} \cap \hat{\beta} \neq \varnothing$. 

We construct $\xi$ in the following way. 

We have $F \e {j+1} \hat{\beta} \cap F \e j \hat{\beta} \neq \varnothing $ for $0 \leq j \leq n-1$ and by Lemma \ref{first case of the lifting} $F \e n \hat{\beta} \cap (\hat{\beta} +(2k\pi,0)) \neq  \varnothing$ for some $k \geq 1$. 

This implies that the union $\hat{\beta} \cup F \hat{\beta} \cup \ldots \cup F \e n \hat{\beta} \cup (\hat{\beta} +(2k \pi , 0))$ is path connected and therefore this union contains a path $\hat{\xi}$
connecting any point $(\theta_{0},r_{0}) \in \hat{\beta}$ to $(\theta _{0} +2k \pi , r_{0})$ for some $k \geq 1$. 

Let $P: \mathbb{R} \times (0,\infty) \rightarrow \mathbb{R} \e 2 \setminus\{(0,0)\}$ be the covering map and $\xi = P( \hat\xi )$.

\begin{Lemma} \label{contradiction}
It follows that $\xi$ is a closed path of $S$ that satisfies the following.
    \begin{enumerate}
    \item $\xi \subset U$.
    \item $\xi \subset B_{\delta } \e \circ -\{0\}$.
    \item If $B'$ is the component of $S \setminus \xi$ which contains $p=0$ then
    \begin{enumerate}
        \item $B' \subset B \e \circ$.
        \item $fr_S B' \subset \xi$.
    \end{enumerate}
\end{enumerate}
\end{Lemma}

\begin{proof}
Since $\beta  \subset U$ and $\xi  \subset  \cup _{j =0}^{n}f^{j} \beta $ we have that $\xi \subset U$, which proves item $(1)$.

Item $(2)$ follows immediately from the construction of $\xi$.
Observe that $\xi$ is a non trivial element of $\pi_{1} (B_{\delta} \e \circ -\{0\})$. 

Let $B'$ be the component of $S \setminus \xi$ which contains $p$.

If $\eta_1 < \inf \{r>0 \mid (\theta,r) \in \hat{\xi} \text{ for some } \theta \in [\theta_0,\theta_0 + 2k \pi] \}$ then $P \e {-1} (B')$ contains $\mathbb{R} \times (0,\eta_1)$.

Similarly, if $\delta > \eta_2 > \sup \{r>0 \mid (\theta,r) \in \hat{\xi} \text{ for some } \theta \in [\theta_0,\theta_0 + 2k \pi] \}$ then $P \e {-1} (B')$ is contained $\mathbb{R} \times (0,\eta_2)$.

This implies that $B_{\eta_1} \e \circ \subset B' \subset B_{\eta_2} \e \circ$. 
Therefore $B'\subset B_{\delta} \e \circ \subset B \e \circ$, which proves item $(3a)$.
Item $(3b)$ is obvious.
\end{proof}

The contradiction is obtained in the following way.

Since $X(e)=\{p\} \subset Z (b)$ we have that $Z (b) \cap B'  \neq \varnothing $. 
On the other hand, since $B' \subset B$ and $Z(b)$ is not contained in $B$, we have that $Z(b) \cap (S-B') \neq \varnothing$. 

Being connected, $Z(b) \cap fr_{S}{B'} \neq \varnothing$ and therefore $fr_{S}{B'} \cap fr_S U \neq \varnothing$.

On the other hand, by Lemma \ref{contradiction}, we have $fr_{S}B' \subset \xi \subset U$.


\vspace{2mm}

\textbf{5) The proof when $F\hat{\beta} \cap (\hat{\beta} + (2\pi,0)) \neq \varnothing$.}
When $F \hat{\beta} \cap (\hat{\beta}+(2\pi,0)) \neq \varnothing$ we work with a different lifting of $f$, $H (\theta,r) =F(\theta,r)-(2\pi,0)$. 

Using equations (\ref{condition on epsilon 2}) and (\ref{estimate fˆj(p) and F'ˆj(0)})
it is not difficult to show that $H$ satisfies $H \hat{\beta} \cap \hat{\beta} \neq \varnothing $ and $H \e n \hat{\beta} \cap (\hat{\beta} +(2k\pi,0)) \neq \varnothing $ for some $k \leq -1$. 

The construction of $\xi $ is done as in the previous case and we arrive at a contradiction in the same way.

This proves item $(1)$ of Theorem \ref{The frontier of domains with fixed prime ends}.



\subsection{The covering \texorpdfstring{$\phi:C(b)\rightarrow Z(b)$}{phi:C(b) rightarrow Z(b)}}

In this subsection we will prove items $(2)$, $(3)$ and $(4)$ Theorem \ref{The frontier of domains with fixed prime ends}. 

Before we begin, it is important to mention that we learned the idea of the cone in Lemma \ref{local semiconjugacy} from John Franks and Patrice Le Calvez in \cite{FrLe}, where they clarified Mather's original idea given in \cite{Mather1981}.

\subsubsection{The accumulation lemma}

Now we will present a very useful result in conservative dynamics on surfaces. 
For a proof see Corollary 8.3 of \cite{Mather1981}.

\begin{Proposition} \label{the accumulation lemma}
\textsc{The Accumulation Lemma}. 
    Let $S$ be a connected surface, $f:S \rightarrow S$ an area preserving homeomorphism of $S$ and $K$ a compact connected invariant set of $f$. If $L$ is a branch of $f$ and $L \cap K \neq  \varnothing $ then $L \subset K$.
\end{Proposition}

Next, we would like to prove a corollary of the Accumulation Lemma that will be needed later.

\begin{Corollary}\label{corollary of the accumulation lemma}
 Let $f:S \rightarrow S$ be an area preserving $C \e 1$ diffeomorphism of a compact connected surface $S$ and $L$ be a branch of a saddle fixed point $p$.
Then we have the following:

\begin{enumerate}
    \item Let $U$ be an invariant domain such that $fr_S U$ has finitely many components.
    If $L \cap U \neq  \varnothing $ then $U$ contains $L$ and the two sectors of $p$ adjacent to $L$.
    \item If $K$ is a compact connected invariant set then either $L \subset K$ or $L$ and its adjacent sectors are contained in one component of $S -K$.
\end{enumerate}
\end{Corollary}

\begin{proof}

Assume by contradiction that $L \not\subset U$.
Then $L$ intersects a component $K$ of $fr_{S}U$, which is a periodic set. 
By the accumulation Lemma applied to a power of $f$ and $K$, we have $L \subset K$, a contradiction. 
Therefore $L \subset U$.
 
Let $x \in L$ and $W \subset U$  be a neighborhood in $U$ of the arc from $x$ to $f \e 2 (x)$ in $L$. 
Then $ \cup _{n \in \mathbb{Z}}f \e {2n} (W)$ contains the sectors of $p$ adjacent to $L$. 

This proves item $(1)$.
$(2)$ follows from $(1)$.  
\end{proof}



\subsubsection{The local semiconjugacy around a fixed prime end}

Recall our assumptions.
$S$ is a compact connected orientable surface and $f:S \rightarrow S$ is an area preserving orientation preserving $C \e 1$ diffeomorphism of $S$.
$U$ is an invariant domain of $S$ such that $fr_S{U}$ has a finite number of connected components. 
$b$ is a regular ideal boundary point of $U$ such that $f_*(b)=b$, $\hat{f}:C(b) \rightarrow C(b)$ is the homeomorphism on the related circle of prime ends, and all fixed points of $f$ in $Z(b)$ are non degenerate. 
We assume that there exists a fixed prime end $e \in C(b)$. 
In the previous subsection, we saw that if $e \in C(b)$ is fixed prime end then $e$ is accessible and its principal point $p$ is an isolated fixed point of $Z(b)$ of saddle type.

We start by taking continuous coordinates $(x,y)$ in a neighborhood $V$ of $p$ with $p$ at the origin, where $f(x,y) =(\lambda x ,\lambda \e {-1} y)$ with $|\lambda| > 1$.
We may assume that $V$ is the open ball $B_1 \e \circ$ of radius $1$ and center at $(0,0)$. 

There exists a path $\beta:(0,1] \rightarrow U$ such that $\underset{t \rightarrow 0}{\lim }\beta (t) = p$ in $S$ and $\underset{t \rightarrow 0}{\lim }\beta (t) = e$ in $E(U)$.
We may assume that $\beta(0,1) \subset V$ and $\beta(1) \notin V$. 

From item Lemma \ref{preparation lemma}, there exists $\delta \in (0,1)$ such that for any decreasing sequence $(r_n)_{n \geq 1}$ contained in $(0,\delta )$ with $\underset{n \rightarrow \infty }{\lim }r_n=0$, there exists a chain $(V_n)$ representing $e$ such that $\xi_n = fr_{U}V_n \subset C_{r_n}$.

Since $\underset{t \rightarrow 0}{\lim }\beta (t) = e$ in $E(U)$ we may assume that $\xi_n \cap \beta \neq \varnothing$ for every $n$, and therefore $\beta(0,1) \cup (\cup_n \xi_n)$ is connected and contained in both $U$ and $V$.
Therefore $\beta$ and all the arcs $\xi_n$ must be contained in one connected component of $U \cap V$.

Denote this component by $W$.

By Corollary \ref{corollary of the accumulation lemma} a branch of $p$ is either contained in $U$ or is disjoint from it.

Recall that our hypotheses do not allow taking powers of $f$.

\begin{Lemma}
    We have that $\lambda > 1$.
\end{Lemma}

\begin{proof}
$W$ can not contain the four local branches, otherwise $Z(b)=\{p\}$ by Corollary \ref{corollary of the accumulation lemma}. 
This contradicts that $b$ is regular.
From this we conclude that one of the local branches is not contained in $W$ and therefore this branch is disjoint from $W$ and from $U$.

Assume by contradiction that $\lambda < {-1}$.
In this case all local branches have period two.

If $U$ is disjoint from, say, $L_1=\{(x,y) \in V \mid x>0 ,y=0\}$ then it is disjoint from the local invariant manifold, $\{(x,y) \in V \mid -1<x<1 ,y=0\}$. 
This local invariant manifold separates $V$ into two "half open" balls, $B_+$ and $B_-$.

If $\beta$ is contained in $B_+$ then $f \circ \beta$ is contained in $B_-$. As $t \rightarrow 0$ both paths tend to $p$ in $S$, but $\beta$ tends to $e$ in $E(U)$ and $f \circ \beta$ tends to $\hat{f} (e)$.

But since $\beta$ is contained in $B_+$ and $f \circ \beta$ is contained in $B_-$, we have that $\hat{f}(e) \neq e$ and $\hat{f} \e 2 (e) = e$, contradicting the fact that $e$ is a fixed prime end.
\end{proof}

\begin{Lemma} \label{local semiconjugacy}
For each fixed prime end $e \in C(b)$ there exists an arc of prime ends $(a_1,a_2)$ around $e$ in $C(b)$ such that $\phi_e:(a_1,a_2) \rightarrow Z(b)$ defined by $\phi_e(a)=Y(a)$ is a function that satisfies the following:
\begin{enumerate}
    \item $Y(a)$ is a one point set for every $a\in (a_1,a_2)$.
    \item $\phi_e(e)=p$ and $\phi_e((a_1,e))$ and $\phi_e((e,a_2))$ are local branches of $p$.
    \item $\big( \phi_e \circ \hat{f} \big) (a) = \big( f \circ \phi_e \big)(a)$ if $a$ and $\hat{f}(a)$ belong to $(a_1,a_2)$.
\end{enumerate}
\end{Lemma}

\begin{proof}

Let us first consider the case in which $W$ does not intersect any of the branches of $p$.
In this case $W$ must be contained in one of the four sectors of $p$ defined by $V$, say $S_1=\{(x,y) \in V \hspace{0.5mm}\vert\hspace{0.5mm} x>0, y>0\}$.
The local branches adjacent to $S_1$ are $L_1$ and $L_2=\{(x,y) \in V \mid x=0 ,y>0\}$.

From Lemma \ref{frontiers of V_i and fV_i intersect for i>=i_0}, there exists $n_0$ such that $\xi_n \cap f(\xi_n) \neq \varnothing $ for $n \geq n_0$.

Lemma \ref{preparation lemma}  let us choose any norm in $\mathbb{R} \e 2$ to work with.
We are going to consider the sup norm, for which a ball $B_{\rho }$ is the square with vertices at $(\pm \rho,\pm \rho)$. 

Since $(r_n)$ is an arbitrary sequence contained in $(0,\delta)$, given any number $\rho \in (0,r_{n_0})$, we may assume that the sequence $(r_n)$ provided by Lemma \ref{preparation lemma} satisfies $\rho = r_m$ for some $m > n_0$.

Let $\Gamma_{\rho }=\big((0,\rho] \times \{\rho\}) \cup (\{\rho\} \times (0,\rho]\big)$. 
Then $\xi_m \subset \Gamma_{\rho}$, and since $f(\Gamma_{\rho}) \cap \Gamma_{\rho}=\{(\rho,\lambda \e {-1} \rho)\}$ we have that $f(\xi_m) \cap \xi_m =\{(\rho,\lambda \e {-1} \rho)\}$ as well. 

Therefore $(\rho,\lambda \e {-1} \rho)$ and $(\lambda \e {-1} \rho,\rho)$ belong to $\xi_m$, and the arc $\Lambda_{\rho}$ from $(\rho,\lambda \e {-1} \rho)$ to $(\lambda \e {-1} \rho,\rho)$ inside $\Gamma_{\rho }$ is contained in $\xi_m \subset U$.
This holds for every $\rho \in (0,r_{n_0})$.

The set $C:=\underset{\rho \in (0,r_{n_0})}{\bigcup} \Lambda_{\rho}$ is the intersection of $B \e \circ_{r_{n_0}}$ with the closed cone with vertex at $(0,0)$ and boundary at the half-lines from $(0,0)$ and slopes $\lambda \e {-1}$ and $\lambda$.

We have that $C \subset U$, and therefore $R:= (0,r_{n_0}) \times (0,r_{n_0}) \subset \underset{n \in \mathbb{Z}}{\bigcup} f \e n C \subset U$.

Since $R \subset V$ and $R$ intersects the arcs $\xi_m$, we conclude that $R \subset W$.

For any $q \in \Big(\big([0,r_{n_0}) \times \{0\}\big) \cup \big(\{0\} \times [0,r_{n_0})\big)\Big)$, if $B_{\frac{1}{n}}\e \circ (q)$ is the open ball of radius $\frac{1}{n}$ with center at $q$, then $V_n = B_{\frac{1}{n}}\e \circ (q) \cap R$ defines a prime chain $(V_n)=a$ such that $Y(a)=\{q\}$.

Let $(a_1,a_2)$ be the arc of prime ends whose impressions are the points of $\big([0,r_{n_0}) \times \{0\}\big) \cup \big(\{0\} \times [0,r_{n_0})\big)$. 
We define $\phi_e(a)=Y(a)$ for $a \in (a_1,a_2)$.

Items $(1)$ and $(2)$ are easy to check.

Item $(3)$ follows from the general fact that for any $e' \in E(U)$ we have the equality of sets $f\big(Y(e')\big)=Y\big(\hat{f}(e')\big)$.

This completes the construction of the local semiconjugacy when $W$ does not intersect any branch of $p$.

We will need item $(1)$ of Corollary \ref{corollary of the accumulation lemma}, which says that if $L$ is a branch of $p$ and $L \cap U \neq \varnothing$ then $U$ contains $L$ and the two sectors of $p$ adjacent to $L$.

Now we consider the case where $W$ intersects one branch of $p$, say the branch that contains the local branch $L_2$, and $W$ does not intersect the branches that contain $L_1$ and $L_3=\{(x,y) \in V \hspace{0.5mm}\vert\hspace{0.5mm} x<0 ,y=0\}$.

From Corollary  \ref{corollary of the accumulation lemma}, we have that $W=S_1 \cup L_2 \cup S_2$, where $S_2=\{(x,y) \in V \hspace{0.5mm}\vert\hspace{0.5mm} x<0, y>0\}$.
It follows that, any $q \in \{(x,y) \in V \hspace{0.5mm}\vert\hspace{0.5mm} -1<x<1, y=0\}$ is the impression of a prime end defined by chains made of "half" open balls centered at $q$ and contained in $W=S_1 \cup L_2 \cup S_2$.
The semiconjugacy is constructed as in the previous case.

Now it should be clear that there are four possibilities for $W$: $W$ consists one sector; $W$ consists of a local branch and its two adjacent sectors; $W$ consists of three sectors and the two local branches between them; $W$ consists of the complement in $V$ of the union of $p$ with one local branch.

As we said before, $W$ can not contain the four local branches.

We described the construction of the local semiconjugacy in the first two cases. In the other two, the construction is analogous to the second case.
\end{proof}

\subsubsection{The construction of the semi conjugacy}

The existence of the local semiconjugacy $\phi_e:(a_1,a_2) \rightarrow Z(b)$ around every fixed prime end $e$ of $C(b)$ implies that they are isolated and exist in a finite number.

Since $\phi_e (e) = p$ and each branch of $p$ is invariant, the existence of the local semiconjugacy also implies that $(a_1,e)$ and $(e,a_2)$ are mapped by $\hat{f}$ onto arcs that satisfy $\hat{f}(a_1,e) \cap (a_1,e) \neq \varnothing$ and $\hat{f}(e,a_2) \cap (e,a_2) \neq \varnothing$.

From this we conclude that $\hat{f}:C(b) \rightarrow C(b)$ is orientation preserving and has a finite number of fixed points. 
This proves item $(2)$ of Theorem \ref{The frontier of domains with fixed prime ends}.

Let $[e_1,e_2]$ be an arc of prime ends, where the end points are fixed, and if $e \in (e_1,e_2)$ then $\underset{n \rightarrow \infty}{lim} \hat{f} \e n (e) =e_2$ and $\underset{n \rightarrow -\infty}{lim} \hat{f} \e n (e) =e_1$.
For $i=1,2$ let $(a_{i1},a_{i2})$ be the arc around $e_i$ where the local semiconjugacy $\phi_{e_i}:(a_{i1},a_{i2}) \rightarrow Z(b)$, $\phi_{e_i}(a)=Y(a)$, is defined.

Let $a \in (e_1,e_2)$.
There exist $n_1$ and $n_2$ such that $f \e {n_i} (a) \in (a_{i1},a_{i2})$.
We have that $f \e n {\big(Y(e')\big)} = Y\big((\hat{f}) \e n (e')\big)$ for any $e'\in E(U)$ and $n \in \mathbb{Z}$.
Since $Y\big(\hat{f} \e {n_i} (a)\big)$ is a one point set, the same happens to $Y(a)$.

Therefore $Y(a)$ is a one point set for every $a \in C(b)$ and $\phi:C(b) \rightarrow Z(b)$, $\phi(a)=Y(a)$, is well defined.

Obviously $f \circ \phi = \phi \circ \hat{f}$ on $C(b)$.

Each arc $(a_{i1},a_{i2})$ is mapped into the union of $p_i=Y(e_i)$ and two of its local branches.
Since $f \e {n_i} (a) \in (a_{i1},a_{i2})$ for $i=1,2$, we have that $Y(a)$ belongs to a branch of each $p_i$. 

From this we conclude that the arc $(e_1,e_2)$ is mapped onto a connection from $p_1$ to $p_2$.

Now we show that $\phi$ is continuous.
Let $a \in C(b)$.

Consider a neighborhood $A \cap Z(b)$ of $\phi(a)=Y(a)$ in $Z(b)$, where $A$ is an open subset of $S$.
If $a$ is represented by a prime chain $(V_n)$ then $Y(a)=\cap_n cl_S{V_n}$, and therefore $cl_S {V_n} \subset A$ for some $n$.
We have that $V'_n$ is a neighborhood of $a$ in $E(U)$.
If $e \in C(b) \cap V'_n$, then $e=(W_k)$ with $W_k \subset V_n$ for some $k$.

It follows that $Y(e) = \cap_k cl_S{W_k} \in A \cap Z(b)$, and $\phi$ is continuous at $a$.

It remains to show that $\phi:C(b) \rightarrow Z(b)$ is surjective. 

This follows from the next result that is probably known to experts, but which, as far as we know, there is no written proof anywhere. 

We would like to remark that Proposition \ref{impression union of impressions} is a general result about any circle of prime ends related to a regular ideal boundary point. There are no extra hipotheses or homeomorphisms involved.

\begin{Proposition} \label{impression union of impressions}
    Let $S$ be a compact connected surface and $U \subset S$ a domain such that $fr_S U$ has finite number of connected components. If $b$ is a regular ideal boundary point of $U$ then $Z(b)= \cup_{e \in C(b)} Y(e)$.
\end{Proposition}

\begin{proof}
We already know that $Y(e) \subset Z(b)$ for every $e \in C(b)$.

Let $x \in Z(b)$ where $b=(P_n) \in b_{reg} (U)$.
We would like to prove that $x \in Y(e)$ for some $e$ in the corresponding circle of prime ends $C(b)$.

Since $Z(b) \subset fr_S U$ we have that $x \notin U$.

We know that $x \in Z(b)$ if and only if there exits a sequence $(x_k)$ in $U$, such that $x_k \rightarrow x$ in $S$ and $x_k \rightarrow b \in b(U)$ in $B(U)$.
We have that $(x_k)$ is a sequence in $E(U)$, which is compact.
Taking a subsequence if necessary, there exists $e \in E(U)$ such that $x_k \rightarrow e \in E(U)$.

We can not have $e \in U$. 
In fact, if $e \in U$ then since $x_k \rightarrow x$ in $S$ we would have $e=x$, contradicting the that $x \notin U$.

By the construction of $E(U)$, for every $n$ we have that $C(b) \sqcup P_n$ is a neighborhood of $C(b)$ in $E(U)$ (see section $10$ of \cite{Mather1982}).
We may assume that this neighborhood contains no other prime ends than those of $C(b)$.
Therefore $e \in C(b)$.

If $(V_i)$ is a chain representing $e$, then for every $i$ there exists $k_i$ such that $x_k \in V_i \subset cl_S V_i$ for $k \geq k_i$.
From this we conclude that $x \in cl_S V_i$ for every $i$ and $x \in Y(e)=\cap_i cl_S V_i $, with $e \in C(b)$.
\end{proof}

This proves items $(3)$ and $(4)$ of Theorem \ref{The frontier of domains with fixed prime ends} and its proof is complete.

\subsubsection{A result of Mather}

Finally, we would like to present a result of Mather under weaker hypotheses.

\begin{Proposition} \label{if p is in K then the invariant manifolds are in K}
Let $S$ be a connected surface, $f:S \rightarrow S$ an area preserving orientation preserving $C \e 1$ diffeomorphism and $K$ be a compact connected invariant set of $f$ such that every fixed point of $f$ in $K$ is non degenerate.

Let $p$ be a saddle fixed point of $f$ and assume that no branch of $p$ is a connection.

If $p \in K$ then $W_p \e u \cup W_p \e s \subset K$.
\end{Proposition}

\begin{proof}
Assume by contradiction that a branch $L$ of $p$ is not contained in $K$.

By item $(2)$ of Corollary \ref{corollary of the accumulation lemma}, $L$ and its adjacent sectors are contained in one component $U$ of $S-K$.

We know that $fr_S U$ is compact and has a finite number of connected components.
Since $L$ is invariant, $U$ is also invariant.
If we parameterize the local branch of $L$ by $\alpha:(0,1] \rightarrow L$ so that $\underset{t \rightarrow 0 \e +}{\lim}\alpha (t) = p$ then, as $t \rightarrow 0 \e +$, we have that $\alpha(t) \rightarrow e$ for some $e \in E(U)$ and $\alpha(t) \rightarrow b$ for some $b \in B(U)$.

Since $Z(b) \subset fr_S U \subset K$, all fixed points of $f$ in $Z(b)$ are non degenerate.
We have that $e$ is a fixed point of $\hat{f}:C(b) \rightarrow C(b)$ and that $X(e)=\{p\} \subset Z(b)$.

By Theorem \ref{The frontier of domains with fixed prime ends} there exists a finite to one semiconjugacy $\phi:C(b) \rightarrow Z(b)$ and $Z(b)$ is a finite union of connections that contains at least one branch of $p$.

This contradicts the fact that no branch of $p$ is a connection.  
\end{proof}



\subsection{Some examples and simple consequences of Theorem \ref{The frontier of domains with fixed prime ends}}

\begin{Corollary} \label{Corollary 1 the frontier of domains with fixed prime ends}

Let $S$ be a compact connected orientable surface and $f:S \rightarrow S$ an area preserving orientation preserving $C \e 1$ diffeomorphism of $S$.
Assume that $U$ is a periodic domain of $S$ such that $fr_S{U}$ has a finite number of connected components and that every ideal boundary point of $U$ is regular. 
Suppose that all periodic points of $f$ in $fr_S U$ are non degenerate.

Let $E(U)$ be the prime ends compactification of $U$ and assume that all induced maps on circles of prime ends have rational rotation number.

Then $fr_S U$ is the union of finitely many saddle connections and the corresponding saddles.
\end{Corollary}

\begin{Corollary} \label{Corollary 2 the frontier of domains with fixed prime ends}
Let $S$ be a compact connected orientable surface and $f$ an area preserving orientation preserving $C \e 1$ diffeomorphism of $S$. Let $U$ be an invariant set homeomorphic to an open disk and suppose that all fixed points of $f$ in $fr_S U$ are non degenerate. If $fr_S U$ has more then one point then $fr_S U$ is the connected union of finitely many saddle connections and the corresponding saddles.
\end{Corollary}

\begin{Example}
Let $H:\mathbb{R} \e 2 \rightarrow \mathbb{R}$ given by $H(x,y)=\sin{(\pi x) \sin{(\pi y)}}$.

We have that $H(x+2k,y+2l)=H(x,y)$ for every $(k,l) \in \mathbb{Z} \e 2$ and $H$ is periodic in both variables with period $2$.

Consider the Hamiltonian flow generated by $H$ and let $f:\mathbb{R} \e 2 \rightarrow \mathbb{R} \e 2$ be its time one map.
We know that $f$ preserves the two dimensional Lebesgue measure.

If we look at the Hamiltonian equations of the flow and the level curves of $H$, then we conclude the following.

\begin{enumerate}
    \item $f(x+2k,y+2l)=f(x,y)$ for every $(k,l) \in \mathbb{Z} \e 2$.
    \item All fixed points are saddles or elliptic.
    \item A fixed point $p$ is a saddle if and only if $p \in \mathbb{Z} \e 2$.
    \item A fixed point is elliptic if and only if $(x,y)=(k+\frac{1}{2},l+\frac{1}{2})$ for $(k,l) \in \mathbb{Z} \e 2$.
    \item The segments $(k,k+1) \times \{l\}$ and $\{k\} \times (l,l+1)$, where $(k,l) \in \mathbb{Z} \e 2$ are all saddle connections.
\end{enumerate}

With these properties in mind we will construct a couple of examples on the torus $\mathbb{T} \e 2$.  

We will think of $\mathbb{T} \e 2$ as the square with vertices at $(\pm 2, \pm 2$), with opposite sides identified.
We will still denote by $f$ as the induced map on $\mathbb{T} \e 2$.

In the first example we consider $K=\big([-1,1] \times \{0\}\big) \cup \big(\{0\} \times [-1,1]\big) $ and $U = \mathbb{T} \e 2 \setminus K $.
We have that $U$ is a surface of genus one that has only one ideal boundary point, $b(U)=\{b\}$, and $B(U)$ is homeomorphic to a torus.

We have that $K$ and $U$ are invariant under $f$. 
$K=fr_S U=Z(b)$ is made of $5$ fixed points of saddle type, $(0,0)$, $(\pm1,0)$ and $(0,\pm1)$,
and $4$ connections, $(-1,0) \times \{0\}$, $(0,1) \times \{0\}$, $\{0\} \times (-1,0)$ and $\{0\} \times (0,1)$.

On the other hand, $E(U)$ is a surface of genus one and one boundary component $C(b)$. 
The restriction of $\hat{f}$ to $C(b)$ has $8$ fixed points.
If we group these fixed points into two groups of 4 points each, so that in a circular order the points of each group alternate, then the semiconjugacy $\phi:C(b) \rightarrow Z(b)$ takes one group into $(0,0)$, and the other 4 fixed points into each one of the remaining saddles of $K$, $(\pm1,0)$ and $(0,\pm1)$. 

Each saddle connection is the image of two arcs between fixed points of $C(b)$.

Now we describe the second example.

Let $C$ be the circle $[-2,2] \times\{0\}$, $I_1=\{0\} \times [0,1]$, $I_2=\{0\} \times [-1,0]$, $K=C \cup I_1 \cup I_2$ and $U=\mathbb{T} \e 2 \setminus K$.

We have that $K$ is compact and connected, $K=fr_S U$, $U$ has two ideal boundary points, $b(U)=\{b_1,b_2\}$ and $B(U)$ is homeomorphic to a sphere.
If we approach $C \cup I_1$ using points of $U$ above $C$ and call this ideal boundary point $b_1$ then $Z(b_1)=C \cup I_1$. Approaching from below we call the ideal boundary point $b_2$ and $Z(b_2)=C \cup I_2$.
Both $Z(b_1)$ and $Z(b_2)$ have $5$ saddles and $5$ connections.

We have that $E(U)$ is a compact surface of genus zero and two boundary components, $C(b_1)$ and $C(b_2)$.
Each circle of prime ends has $6$ fixed points of $\hat{f}$.

Each semi conjugacy $\phi_i:C(b_i) \rightarrow Z(b_i)$ maps two fixed points of $C(b_i)$ into $(0,0)$ and the other fixed points of $C(b_i)$ are mapped into one fixed point of $Z(b_i)$ each.
$I_i$ contains the only connection of $Z(b_i)$ which is the image of two arcs between fixed points of $C(b_i)$.
\end{Example}



\section{The accumulation of branches and homoclinic points}

Now we present some results about the accumulation of invariant manifolds and the existence of homoclinic points in low genus.

If $L$ is a branch of a saddle, we will denote by $L[q_1,q_2]$ the arc inside $L$ with ends points $q_1$ and $q_2$. 
We will need the following result that appeared in \cite{Ol2}, but which was already known by experts since \cite{Ol1}.

\begin{Proposition} \label{accumulation results}
Let $S$ be a connected compact orientable surface, $f:S \rightarrow S$ an area preserving orientation preserving $C \e 1$ diffeomorphism of $S$. Let $p$ be a saddle fixed point of $f$ and assume that the branches of $p$ are invariant.
\begin{enumerate}

\item Let $L$ be a branch of $p$ and suppose that all fixed points of $f$ contained in $cl_{S}L$ are non degenerate.
Then either $L$ is a connection or $L$ accumulates on both sectors adjacent to itself. 
    
\item Assume that the four branches of $p$ are not connections and that all fixed points of $f$ contained in $cl_{S}(W_p \e u \cup W_p \e s)$  are non degenerate. 
Then we know the following.
    \begin{enumerate}
        \item The four branches of $p$ have the same closure in $S$.
        \item If $S$ is the sphere or the torus then all pairs of adjacent branches of $p$ intersect.
    \end{enumerate}

\end{enumerate}
\end{Proposition}


We will give a brief sketch of the proof of item $(1)$ using our Theorem \ref{The frontier of domains with fixed prime ends}. Then we dedicate the rest of the section to give a new proof of item $(2b)$ in the case of the torus. The arguments in this proof will be essential in the proof of our main result about the Standard Map.


Let $(x,y)$ be continuous coordinates in a neighborhood $V$ of $p$, with $p$ at the origin and where  $f(x,y) =(\lambda x ,\lambda \e {-1} y)$, where $\lambda>1$.
We may assume that $L$ is the branch that contains the local branch $\{(x,y) \in V \hspace{0.5mm}\vert\hspace{0.5mm} x>0 ,y=0\}$.

We are going to prove that if $L$ does not accumulate on one of its adjacent sectors then it is a connection.
Suppose that $L$ does not accumulate on, say,
$\Sigma_1=\{(x,y) \in V  \hspace{0.5mm}\vert\hspace{0.5mm} x>0, y>0\}$. 
So, making $V$ smaller if necessary, we can assume that $cl_S L\cap \Sigma_1 = \varnothing$.
Let $U$ be the component of $S \setminus cl_S L$ that contains $\Sigma_1$. 
Obviously $U$ is invariant.
Since $cl_S L$ is compact and connected, by Proposition \ref{finitely many components} we have that $fr_S U$ has finitely many connected components and $U$ has a finite number ideal boundary points.

If $\beta :(0,1) \rightarrow U$ is defined by $\beta(t)=(t,t)$, then as $t \rightarrow 0 \e +$ we have that $\beta(t) \rightarrow p$ in $S$, $\beta(t) \rightarrow e$ for some $e \in E(U)$ and $\beta(t) \rightarrow b$, for some ideal boundary point of $U$.
Since $ (f\circ \beta)(t) \rightarrow e$ as $t \rightarrow 0 \e +$, we have that $e$ is a fixed point of $\hat{f}:C(b) \rightarrow C(b)$.
By item $(4)$ of Theorem \ref{The frontier of domains with fixed prime ends} we have that $L$ is a connection contained in $Z(b)$.
This proves item $(1)$.

The proof of $(2a)$ follows from $(1)$ and the accumulation lemma.
The proof of item $(2b)$ in the case of the sphere is well known. It first appeared in \cite{Ol1}. See also \cite{Ol2} and \cite{FrLe}.

The proof of item $(2b)$ in the case of the torus has two parts. 
The first is to show that $p$ has homoclinic points. The second is to show that, once it has one homoclinic point, then every pair of adjacent branches of $p$ intersect.

The proof that $p$ has at least one homoclinic point can be done with the arguments that appeared in \cite{Ol1}. 
It can also de done with an idea that appeared in \cite{SaLe} of taking suitable coverings. 
There they are interested in surfaces of genus greater than one, see Lemmas $5.2$, $5.3$ and $5.4$ of \cite{SaLe}.
But if we adapt the argument and take a 3 fold covering of the torus, we obtain another torus with three fixed points, and then it is easy prove the existence of homoclinic points. See \cite{Ol2} for details.

Therefore we are only going to prove the if $p$ has a homoclinic point then the four branches of $p$ have homoclinic points. 

This is true for surfaces of arbitrary genus.



\begin{Proposition} \label{homoclinic 2}
Let $S$ be a connected compact orientable surface of genus $g$ and $f: S \rightarrow S$ an orientation preserving area preserving $C \e 1$ diffeomorphism of $S$. Let $p$ be a saddle fixed point of $f$ and assume that the branches of $p$ are invariant.

Assume that the four branches of $p$ are not connections and that all fixed points of $f$ contained in $cl_{S}(W_p \e u \cup W_p \e s)$ are non degenerate.

Then we know the following:
\begin{enumerate}
    \item If $g=1$ then $p$ has homoclinic points.
    \item For any $g\geq 1$, if $p$ has homoclinic points then all pairs of adjacent branches of $p$ intersect.
\end{enumerate}
\end{Proposition}

As we said above, the proof of item $(1)$ is known.
Now, we would like to prove that for any $g\geq 1$, if $p$ has homoclinic points then all pairs of adjacent branches of $p$ intersect.

We need the following result:

\begin{Lemma} \label{disconnecting the thesis}
Let $M$ be a compact orientable surface of genus $g$ and $\alpha_0,...,\alpha_{2g}$ be a collection of $2g+1$ closed curves of $M$ such that for every $i=0,...,2g$ we have that $\alpha_i \setminus \underset{j \neq i}{\bigcup}\alpha_j \neq \varnothing$. 
Then $M \setminus \underset{0 \leq i \leq 2g}{\bigcup}\alpha_i $ is disconnected.
\end{Lemma}

\begin{proof}
Assume by contradiction that $M \setminus \underset{0 \leq i \leq 2g}{\bigcup}\alpha_i $ is connected.
For $0\leq i \leq 2g$, let $x_i \in \alpha_i \setminus \underset{j \neq i}{\bigcup}\alpha_j$ and $D_i$ a small disk containing $x_i$ and disjoint from $\underset{j \neq i}{\bigcup}\alpha_j$.

Let $y_i$ and $z_i$ be points of $D_i$ on different sides of $\alpha_i$.
Let $\beta_{1i}$ be a small curve connecting $y_i$ and $z_i$ inside $D_i$.
Since $M \setminus \underset{0 \leq i \leq 2g}{\bigcup}\alpha_i $ is connected, there exists a simple curve $\beta_{2i}$ connecting $z_i$ to $y_i$ inside $M \setminus \underset{0 \leq i \leq 2g}{\bigcup}\alpha_i $.
If $\beta_i=\beta_{1i}*\beta_{2i}$, then $\beta_i$ is a simple closed curve with the following properties:

\begin{enumerate}
\item The oriented intersection number $\#(\alpha_i, \beta_i)=1$, if we choose suitable orientations for $M$, $\alpha_i$ and $\beta_i$.
\item $\#(\alpha_j, \beta_i)=0$ for $j \neq i$, since $\beta_i \cap \big(\underset{j \neq i}{\bigcup}\alpha_j \big) = \varnothing$ .
\end{enumerate}

Let $h_i:H_1(M,\mathbb{R}) \rightarrow \mathbb{R}$ be the homomorphism $h_i(\gamma)=\#(\gamma,\beta_i)$.
It follows from $(1)$ and $(2)$ that $h_0,...,h_{2g}$ are linearly independent elements of $\big(H_1(M,\mathbb{R})\big) \e *$, contradicting the fact that the dimension of $\big(H_1(M,\mathbb{R})\big) \e *$ is $2g$.
\end{proof}

Returning to the proof of item (2) of Proposition \ref{homoclinic 2}, we know that two branches of $p$, say $W_+ \e u$ and $W_+ \e s$, intersect at a homoclinic point $q$.

\begin{Lemma} \label{components}
Let $S$ be a connected compact orientable surface of genus $g$ and $f: S \rightarrow S$ an orientation preserving area preserving $C \e 1$ diffeomorphism of $S$. Let $p$ be a saddle fixed point of $f$ and $q \in  W_+ \e u \cup W_+ \e s$. Assume that the branches of $p$ are invariant and that they are not connections.

Let 
\begin{itemize}
    \item $\alpha_0 = W_+ \e u [q,f(q)] \cup W_+ \e s [q,f(q)]$. 
    \item $\alpha_i=f\e i({\alpha_0)}= W_+ \e u [f \e i (q),f \e {i+1} (q)] \cup W_+ \e s [f \e i (q),f \e {i+1} (q)]$, for $0 \leq i \leq 2g$.
    \item $\Omega_u =  W_+ \e u [q,f \e {2g+1} (q)]$ and $\Omega_s =  W_+ \e s [q,f \e {2g+1} (q)]$.
    \item $\Omega=\bigcup_{0 \leq i \leq 2g} \alpha_i=\Omega_u \cup \Omega_s$. 
\end{itemize}

Then we know the following.

\begin{enumerate}
    \item $\alpha_0$ is a closed curve and $S \setminus \Omega$ is disconnected.
    \item The frontier of every component of $S \setminus \Omega$ intersects both $\Omega_u$ and $\Omega_s$.
    \item If $A$ is the component of $S \setminus \Omega$ that contains $p$ and $B$ is any other component of $S \setminus \Omega$ then  $W_- \e u$ and $W_- \e s$ leave $A$ and intersect $B$.
    \item $W_- \e u \cap W_+ \e s \neq \varnothing$ and $W_- \e s \cap W_+ \e u  \neq \varnothing$
\end{enumerate}
    
\end{Lemma} 

\begin{proof}

For $0 \leq i \leq 2g$, $\alpha_i=f\e i({\alpha_0)}= W_+ \e u [f \e i (q),f \e {i+1} (q)] \cup W_+ \e s [f \e i (q),f \e {i+1} (q)]$ are closed curves.

There exists a point $\bar{q} \in W_+ \e u$ such that $\bar{q} \notin W_+ \e s$.
We have that $o(\bar{q}) \subset W_+ \e u$ and $o(\bar{q}) \cap W_+ \e s = \varnothing$.
We may assume that $\bar{q} \in W_+ \e u (q,f(q))$.
For every $k \in \mathbb{Z}$ we have that $f \e k \bar{q} \in  W_+ \e u (f \e i (q),f \e {i+1} (q)) $ if and only if $k=i$.
Since $f \e k \bar{q} \notin W_+ \e s$, for every $k \in {0,...,2g}$ we have that $f \e k \bar{q} \in \alpha_k \setminus \underset{0 \leq j \leq 2g, j \neq k}{\bigcup}\alpha_j$, and $\alpha_0$,...,$\alpha_{2g}$ satisfy the hypothesis of Lemma \ref{disconnecting the thesis}.

It follows that $S \setminus \Omega$ is disconnected. 
Since $\Omega_u$ and $\Omega_s$ are simple arcs homeomorphic to a closed interval, we have that the frontier in $S$ of every component of $S \setminus \Omega$ has points in $\Omega_u \subset W_+ \e u$ and in $\Omega_s \subset W_+ \e s$.

Let $A$ be the component of $S \setminus \Omega$ that contains $p$ and its local branches, and $B$ any other component.  
By recurrence, there exist infinitely many values of $n$ such that $f \e n B \cap B \neq \varnothing$.

Since $fr_S B \cap W_+ \e s \neq \varnothing$, we have that $f \e n B \cap A \neq \varnothing$ for every $n$ large enough.
Hence $f \e n B \cap fr_S A \neq \varnothing$, implying that $B$ contains points of $W_+ \e u$ or $W_+ \e s$.

We have that the four branches have the same closure in $S$. 
This forces $W_- \e u$ and $W_- \e s$ to leave $A$ and enter $B$, by intersecting $fr_S A$ which is contained in $W_+ \e u \cup W_+ \e s$.
Therefore $W_- \e u$ must intersect $W_+ \e s$ and $W_- \e s$ must intersect $W_+ \e u$.
    
\end{proof}

It remains to show that  $W_- \e u \cap W_- \e s \neq \varnothing$. But this can be achieved by repeating the argument starting with a homoclinic intersection of $W_- \e u$ with $W_+ \e s$  or of $W_- \e s$ with $W_+ \e u$.




\section{Real analytic diffeomorphisms and the standard map}

The standard map is a one parameter family of area preserving diffeomorphism of the two dimensional torus $T \e 2 =\mathbb{R} \e 2/\mathbb{Z}\e 2$ given by 

\begin{center}
$f_{\mu}(x,y)=(x+y+\frac{\mu}{2\pi} \sin (2\pi x),y+\frac{\mu}{2\pi}\sin (2\pi x))$, $\mu \in \mathbb{R}$.
\end{center}

Since $\Phi(x,y) =(\frac{1}{2}-x,-y)$ is a conjugacy between $f_{\mu}$ and $f_{-\mu}$ and $f_0$ is just a linear twist, we only consider parameters $\mu>0$.

For $\mu \neq 0$ there are two fixed points, $p=(0,0)$ and $q =(\frac{1}{2},0)$.  
$p$ is always a saddle with positive eigenvalues and it is called \textit{the principal fixed point} of $f_\mu$.
$q$ is elliptic if $0 <\mu <4$ and a hyperbolic fixed point with negative eigenvalues if $\mu>4$.

In this section we are going to prove the following.

\begin{Theorem} \label{entropy for the standard map}
If $\mu \neq 4$ then the standard map satisfies the following:
\begin{enumerate}
    \item The four branches of $p=(0,0)$ have topologically transverse homoclinic points.
    \item $f_\mu$ has positive topological entropy.
    \item There exist periodic saddles with transverse homoclinic points.
\end{enumerate}
\end{Theorem}




\subsection{The principal point \texorpdfstring{$p$}{p} always has homoclinic points}

For $\mu \neq 4$, if $L$ is a branch of $p$ then $L$ is invariant and all fixed 
points of $f$ contained in $cl_{S}L$  are non degenerate.

We have that $f(-x,-y)=-f(x,y)$. Therefore the invariant manifolds of $p$ are symmetric with respect to $(0,0)$ and if $L$ is a connection then so is $-L$.

Since $q$ is either elliptic or a saddle whose branches have period two, there is no connection between $p$ and $q$. 
Therefore, if one of the branches of $p$ is a connection, then that connection is equal to two branches of $p$. By symmetry, the invariant manifolds of $p$ are formed of two connections of homoclinic points.

If no branch of $p$ is a connection then by Proposition \ref{homoclinic 2} the four branches of $p$ still have homoclinic points.
Therefore we have the following:

\begin{Proposition}
    For $\mu \neq 4$, either the four branches of $p$ form two connections of homoclinic points symmetric with respect $(0,0)$, or no branch of $p$ is a connection and all pairs of adjacent branches of $p$ intersect.
\end{Proposition}

Next, we are going to show that the first alternative can not happen.



\subsection{No branch of \texorpdfstring{$p$}{p} is a connection}

Now we present Ushiki's theorem.

Let $F_\mu:\mathbb{R}\e2\rightarrow\mathbb{R}\e2$,
$F_{\mu}(x,y)=(x+y+\frac{\mu}{2\pi} \sin (2\pi x),y+\frac{\mu}{2\pi}\sin (2\pi x))$, be the lifting of $f_\mu$ to the universal cover.
Since a connection $L$ of $p=(0,0)$ in $\mathbb{T}\e2$ connects $p$ to itself, a lifting of $L$ would be a connection $\widehat{L}$ connecting two points of $\mathbb{Z}\e2$. 

\begin{Proposition} \label{Ushiki}
    \textsc{(Ushiki's theorem)} If a real analytic diffeomorphism $F:\mathbb{R}\e2 \rightarrow \mathbb{R}\e2$ extends into an automorphism $H:\mathbb{C}\e2 \rightarrow \mathbb{C}\e2$ of the complex space of dimension two, then $F$ has no connections.
\end{Proposition}

This is Theorem $1$ of \cite{Us}.
It follows from basic classification results of complex one dimensional manifolds.
Let $H_\mu:\mathbb{C}\e2\rightarrow\mathbb{C}\e2$,
\begin{center}
    $H_{\mu}(z_1,z_2)=(z_1+z_2+\frac{\mu}{2\pi} \sin (2\pi z_1),z_2+\frac{\mu}{2\pi}\sin (2\pi z_1))$
\end{center}
be an extension of $F_\mu$ to $\mathbb{C}\e2$.
If we write $(w_1,w_2)=H_{\mu}(z_1,z_2)$, then a simple calculation shows that
\[\begin{cases}
   {z_1}={w_1}-{w_2}\\
   z_2=w_2-\frac{\mu}{2\pi}{\sin\big(2\pi(w_1-w_2)\big)}.
\end{cases}\]
It follows that $H_\mu$ is an automorphism of $\mathbb{C}\e2$, and from Proposition \ref{Ushiki} we have that $F_\mu$ has no connections.

This proves the following.

\begin{Proposition}
    For $\mu \neq 4$, $f_\mu$ can not have connections.
\end{Proposition}



\subsection{All pairs of adjacent branches of \texorpdfstring{$p$}{p} intersect topologically transversely}




In this subsection we are going to prove the following.

\begin{Proposition} \label{transverse intersection}
Let $\mu \neq 4$. For every pair of adjacent branches $L \e u$, $L \e s$ of $p$ there exists $q \in L \e u \cap L \e s$ and a small open disk $D$ containing $q$ such that, if $\gamma_u$ and $\gamma_s$ are the components of $L \e u \cap D$ and $L \e s \cap D$ that contain $q$, respectively, then they satisfy the following:

    \begin{enumerate}
        \item $\gamma_u \cap \gamma_s = \{q\}$.
        \item $D \setminus \gamma_u$ and $D \setminus \gamma_s$ have two components.
        \item The components of $\gamma_u \setminus \{q\}$ are contained in different components of $D \setminus \gamma_s$.
        \item The components of $\gamma_s \setminus \{q\}$ are contained in different components of $D \setminus \gamma_u$.
    \end{enumerate}
\end{Proposition}

We know from item $(2)$ of Proposition \ref{homoclinic 2} that all pairs of adjacent branches of $p$ intersect. 

We are going to work the case where we start from $q\in W_+ \e u \cap W_+ \e s$, and then show that the pairs ($W_- \e u$ , $W_+ \e s$) and ($W_- \e s$ , $W_+ \e u$) have a topologically transverse intersection. 
Then, repeating the argument starting with $q\in W_- \e u \cap W_+ \e s$, it is possible to show that the other two pairs ($W_+ \e s$, $W_+ \e u$) and ($W_- \e u$ , $W_- \e s$) have a topologically transverse intersection. 

We are only going to prove that ($W_- \e s$, $W_+ \e u$) have a topologically transverse intersection. 
The argument for ($W_- \e u$, $W_+ \e s$) is analogous. 

From now on we refer to sets and results of Lemma \ref{components}.

Let $\Omega_u = W_+ \e u [q, f \e 3 q]$, $\Omega_s = W_+ \e s [q, f \e 3 q]$ and $\Omega = \Omega_u \cup \Omega_s$. 

Since $p$ has no connections, we have that $T \e 2 \setminus \Omega$ is disconnected.
The frontier in $S$ of every component of $S \setminus \Omega$ is contained in $\Omega$ and intersects both $\Omega_u \subset W_+ \e u$ and $\Omega_s \subset W_+ \e s$.

Let $A$ be the component of $T \e 2 \setminus \Omega$ that contains $p$.
The local branch of $W_- \e s$ must leave $A$, intersect $\Omega$ (at at least one point) and then enter a different component of $T \e 2 - \Omega$. 
Let $C$ be the union of all components of $T \e 2 \setminus \Omega$ different from $A$.

We denote the eigenvalues of $df_p$ by $\lambda>1$ and $\lambda \e {-1}$.
There exist real analytic immersions \[\alpha:\mathbb{R} \rightarrow W_p \e u \subset T \e 2 \hspace{2mm} and \hspace{2mm}\beta:\mathbb{R} \rightarrow W_p \e s \subset T \e 2\] such that
    \begin{enumerate}
            \vspace{2mm}
        \item $f\big(\alpha(t)\big)=\alpha(\lambda t) \hspace{1.5mm} \forall t \in \mathbb{R}$, $\alpha(0)=p$ and $df_p \big(\alpha \e \prime (0)\big)= \lambda  \thinspace \alpha \e \prime (0) $.
           \vspace{2mm}
        \item $f\big(\beta(s)\big)=\beta(\lambda \e {-1} s) \hspace{1.5mm} \forall s \in \mathbb{R}$, $\beta(0)=p$ and $df_p \big(\beta \e \prime (0)\big)= \lambda  \e {-1} \thinspace \beta \e \prime (0) $.
           \vspace{2mm}
    \end{enumerate}  

We may assume that $\beta(0,+\infty)=W_- \e s$.
Note that $\beta(0,+\infty)$ can not intersect $\Omega_s \subset W_+ \e s$.
Since all pairs of adjacent branches of $p$ intersect, we have that $\beta(0,+\infty)$ intersects only $\Omega_u$.

For every positive and sufficiently small $s$, we have that $\beta\big([0,s]\big) \subset A$.
We also know that $\beta(s) \in C$ for some values of $s>0$.
Therefore the set 
\begin{center}
    $E=\{s \geq 0 \mid \beta\big([0,s]\big) \subset A \cup \Omega_u\}$
\end{center}
is not empty and bounded above. 
Let $\bar{s} := sup E$.
It easy to see that $E$ is the interval $[0,\bar{s}]$ and that $\beta(\bar{s}) \in \Omega_u$.

We will need the following.

\begin{Lemma} \label{crossing}
There exists $\varepsilon>0$ such that 
$\beta(\overline{s}-\varepsilon,\overline{s}) \subset A$ and 
$\beta(\overline{s},\overline{s}+\varepsilon) \subset C$.
\end{Lemma}

\begin{proof}

We are going to work with the lift to $\mathbb{R} \e 2$ of various objects we are dealing with. We will use the same notation for an object downstairs and for its lift to the universal cover.

Firstly, we will make a real analytical change of coordinates at $p=(0,0)$, so that the local invariant manifolds also become the coordinate axes.
Consider the mapping $\Phi:\mathbb{R} \e 2 \rightarrow \mathbb{R} \e 2$, given by $\Phi(t,s) = \alpha(t) + \beta(s)$.
We have that
\[\Phi \e \prime (0,0) =
 \begin{bmatrix}
 \alpha_1 \e \prime (0) & \beta_1 \e \prime (0) \\
 \alpha_2 \e \prime (0) & \beta_2 \e \prime (0) \\
 \end{bmatrix},
\]
and therefore $det \hspace{0.5mm} \Phi \e \prime (0,0) \neq 0$, since $\alpha \e \prime (0)$ and $\beta \e \prime (0)$  are eigenvectors of $df_p$ corresponding to different eigenvalues.

There exist neighborhoods $\mathcal{U}$ and $\mathcal{V}$ of $(0,0)$ such that the restriction $\Phi: \mathcal{U} \rightarrow \mathcal{V}$ is a real-analytic diffeomorphism.
Its inverse is also real-analytic diffeomorphism, and will be denoted by $\Psi:\mathcal{V} \rightarrow \mathcal{U}$.
We will also assume that $\mathcal{U}$ is the ball $B_\delta$ with center at $p$ and radius $\delta$, where $\delta$ is small.

The local invariant manifolds defined by $\mathcal{V}$, $W_{loc} \e u$ and $W_{loc} \e s$, are the components of $W \e u \cap \mathcal{V}$ and $W \e s  \cap \mathcal{V}$, respectively, that contain $p=(0,0)$.

We have that $\Psi\big(\alpha(t)+\beta(s)\big)=(t,s), \thinspace \forall \thinspace (t,s) \in B_\delta$.
Therefore, for $t,s \in (-\delta,\delta)$ we have that
\begin{center}
    $\Psi$\big($\alpha(t)\big)=(t,0)$ \hspace{2mm} and \hspace{2mm} $\Psi$\big($\beta(s)\big)=(0,s)$,
\end{center} 
meaning that $\alpha$ and $\beta$ could be thought as coordinate curves in the original coordinates. 
If $\Psi_i$, $i=1,2$, are the coordinate functions of $\Psi$ then 
\begin{center}
$W_{loc} \e u = \{(x,y) \in \mathcal{V} \thinspace | \thinspace \Psi_2(x,y)=0\}$ and $W_{loc} \e s = \{(x,y) \in \mathcal{V} \thinspace | \thinspace \Psi_1(x,y)=0\}$.
\end{center}

We may assume that $f \e i q \in W_{+loc} \e u$ for $0 \leq i \leq 3$ and therefore $\Omega_u \subset W_{+loc} \e u$.
It follows that $\Psi(\Omega_u) \subset \{(t,s) \in B_\delta \mid s=0\}$.

We have that $(\Psi_2 \circ \beta)(\bar{s})=0$.
Therefore there exists $\delta > 0$ such that $\beta(\bar{s}-\delta,\bar{s}+\delta) \subset \mathcal{V}$. 
It follows that the composition $\Psi_2 \circ \beta:(\bar{s}-\delta,\bar{s}+\delta) \rightarrow \mathbb{R}$ is a well defined real analytic function.

For $s \in (\bar{s}-\delta,\bar{s}+\delta)$ we have that $(\Psi_2 \circ \beta)(s)=0$ if and only if $\beta(s) \in \Omega_u$.

Since the zeroes of a non constant real analytic function of one variable are isolated, we know that there exists $\varepsilon \in (0,\delta)$ such that $\Psi_2 \circ \beta:(\bar{s}-\varepsilon,\bar{s}+\varepsilon) \rightarrow \mathbb{R}$ satisfies exactly one of the following conditions.

\begin{enumerate}
    \item $(\Psi_2 \circ \beta)(s)=0$ for every $s \in (\bar{s}-\varepsilon,\bar{s}+\varepsilon)$.
    \item If $s \in (\bar{s}-\varepsilon,\bar{s}+\varepsilon)$ then $(\Psi_2 \circ \beta)(s)=0$ if and only if $s=\bar{s}$.
\end{enumerate}

Since $\bar{s}=sup{E}$, we know that for every sufficiently large $n$ there exists $s_n$ such that $\bar{s}<s_n<\bar{s}+\frac{1}{n}$ and $\beta(s_n) \in C$.
This implies that $\beta(s_n) \notin \Omega_u$ and $(\Psi_2 \circ \beta)(s_n) \neq 0$ for every $n$.
From this we conclude that condition $(1)$ can not happen.

From condition $(2)$ we have that if $s \in (\bar{s}-\varepsilon,\bar{s})$ or $s \in (\bar{s},\bar{s}+\varepsilon)$, then $\Psi_2(\beta(s)) \neq 0$, which implies that $\beta(s) \notin \Omega_u$.
If $s \in (\bar{s}-\varepsilon,\bar{s})$ then $\beta(s) \in A$ and if $s \in (\bar{s},\bar{s}+\varepsilon)$ then $\beta(s) \in C$. 

This proves Lemma \ref{crossing}.
\end{proof}

Let $D$ be a disk that contains $q:=\beta(\bar{s})$ and is small enough so that each of the arcs $\beta(\bar{s}-\varepsilon,\bar{s}+\varepsilon)$ and $\Omega_u$ separate $D$ into two components.
Let $\gamma_s$ and $\gamma_u$ be the components of $D \cap \beta(\bar{s}-\varepsilon,\bar{s}+\varepsilon)$ and $D \cap \Omega_u$, respectively, which contain $q$.

Then it is easy to verify that $D$, $\gamma_s$ and $\gamma_u$ satisfy the conditions $(1)$ to $(4)$ of Proposition \ref{transverse intersection}.

This proves item $(1)$ of Theorem \ref{entropy for the standard map}.



\subsection{Positive topological entropy and transverse homoclinic points}

Topologically transverse homoclinic points force the dynamics to have positive topological entropy.

Let $f$ be a $C \e 1$ diffeomorphism of a surface $S$.

\begin{Proposition}
Let $p$ be a hyperbolic periodic point of $f$ of saddle type. 
Assume that two branches of $p$ have a topologically transverse homoclinic point. 
    
Then some power of $f$ has the full shift on two symbols as a topological factor, i.e., there is a subset $\Lambda$ that is invariant under $f \e n$ for some $n>1$, and a continuous map $\pi:\Lambda \rightarrow \Sigma$, that is onto but not necessarily injective, such that $\pi \circ f \e n = \sigma \circ \pi$, where $ \sigma: \Sigma \rightarrow \Sigma$ is the full shift on two symbols.
\end{Proposition}

We would like to point that the topological crossing could be of infinite order.

See Theorems $2.1$, $2.4$ and Lemma $2.8$ of \cite{BuWe} for a proof. 
The idea boils down to finding a rectangle and strips that are stretched inside the rectangle, as in the usual horseshoe.

From basic properties of topological entropy we have
\begin{center}
    $h_{top} (f) \geq \frac{1}{n} h_{top} (f \e n | \Lambda) \geq \frac{1}{n} h_{top} (\sigma) >0$.
\end{center}

Therefore we have proved item $(2)$ of Theorem \ref{entropy for the standard map}.

\begin{Proposition}
    For $\mu \neq 4$, the standard map has positive topological entropy.
\end{Proposition}

A well-known result of Katok, shows that in the case of $C \e {1+\alpha}$ surface diffeomorphisms, the positiveness of the entropy is equivalent to the existence of transverse homoclinic points. 

From this we have item $(3)$ of Theorem \ref{entropy for the standard map}.

\begin{Proposition}
    For $\mu \neq 4$, the standard map has saddles with transverse homoclinic points.
\end{Proposition}

This completes the Proof of Theorem \ref{entropy for the standard map}.



\subsection{Positive topological entropy for real analytic diffeomorphisms}

Finally, we would like to remark that the arguments presented in this section can be used to prove the following result.

\begin{Proposition}

Let $S$ be a connected orientable real analytic surface of finite genus $g$ and $f:S \rightarrow S$ be an orientation preserving real analytic diffeomorphism that preserves a finite measure $\mu$ which is positive on open sets.

Let $p$ be a saddle fixed point of $f$ that satisfies the following conditions.
\begin{enumerate}
    \item $p$ has no connections.
    \item The four branches of $p$ are relatively compact in $S$.
    \item Every fixed point of $f$ in $cl_S ({W_p \e u \cup W_p \e s})$ is non degenerate.
    \item $p$ has homoclinic points (which happens automatically if $g=0$ or $1$).

\end{enumerate}

Then $p$ has topologically transverse homoclinic points and $h_{top} (f) > 0$.
    
\end{Proposition}

The idea is to consider the ideal completion $B(S)$ of $S$, which in this case is a compact surface that contains $S$ as an open and dense subset. The ideal boundary is totally disconnected and the finite measure extends to $B(S)$. Then we do the same reasoning.


 
\section*{Declarations}

\begin{itemize}
\item Funding: not applicable.
\item Conflict of interest/Competing interests: not applicable.
\item Ethics approval: not applicable. 
\item Consent to participate: not applicable.
\item Consent for publication: not applicable.
\item Availability of data and materials: not applicable.
\item Code availability: not applicable.
\item Authors' contributions: not applicable.
\end{itemize}
\noindent



\def\cprime{$'$} \def\cprime{$'$} \def\cprime{$'$} \def\cprime{$'$}
\providecommand{\bysame}{\leavevmode\hbox to3em{\hrulefill}\thinspace}
\providecommand{\MR}{\relax\ifhmode\unskip\space\fi MR }
\providecommand{\MRhref}[2]{%
  \href{http://www.ams.org/mathscinet-getitem?mr=#1}{#2}
}
\providecommand{\href}[2]{#2}

\end{document}